\documentclass[11pt]{amsart}
\usepackage{geometry}                
\usepackage{graphicx}
\usepackage{amssymb}
\usepackage{epstopdf}
\usepackage{amsmath}
\usepackage{color}
\usepackage{hyperref}
\usepackage{pdfsync}
\usepackage{tikz}
 
\usepackage[all]{xy}
\xyoption{matrix}
\xyoption{arrow}

\newtheorem{thm}{Theorem}[section]
\newtheorem{lem}[thm]{Lemma}
\newtheorem{cor}[thm]{Corollary}
\newtheorem{prop}[thm]{Proposition}
\newtheorem{conj}[thm]{Conjecture}
 
\theoremstyle{definition}
\newtheorem{defn}[thm]{Definition}
\newtheorem{eg}[thm]{Example}

\theoremstyle{remark}
\newtheorem{rem}[thm]{Remark}
 
\numberwithin{equation}{section}
 
 \tikzset{help lines/.style={step=#1cm,very thin, color=gray},
help lines/.default=.5} 
\tikzset{thick grid/.style={step=#1cm,thick, color=gray},
thick grid/.default=1} 

\newcommand{\mat}[1]{\ensuremath{
\left[\begin{matrix}#1
\end{matrix}\right]
}}
 
\newcommand{\vs}[1]{\vskip .#1 cm} 

\newcommand{\then}{\Rightarrow}
\newcommand{\ifff}{\Leftrightarrow}
 
\newcommand{\into}{\hookrightarrow}
 \newcommand{\onto}{\twoheadrightarrow}

\DeclareMathOperator{\Hom}{Hom}%
\DeclareMathOperator{\Ext}{Ext}%
\DeclareMathOperator{\End}{End}%
\DeclareMathOperator{\undim}{\underline{dim}}
 
\newcommand{\field}[1]{\mathbb{#1}}
\newcommand{\ZZ}{\ensuremath{{\field{Z}}}}

\newcommand{\RR}{\ensuremath{{\field{R}}}}
\newcommand{\QQ}{\ensuremath{{\field{Q}}}}
\newcommand{\NN}{\ensuremath{{\field{N}}}}

\newcommand{\commentout}[1]{}

\newcommand{\cB}{\ensuremath{{\mathcal{B}}}}
\newcommand{\cC}{\ensuremath{{\mathcal{C}}}}

\newcommand{\cF}{\ensuremath{{\mathcal{F}}}}

\newcommand{\cG}{\ensuremath{{\mathcal{G}}}}

\newcommand{\cP}{\ensuremath{{\mathcal{P}}}}

\newcommand{\cU}{\ensuremath{{\mathcal{U}}}}
\newcommand{\cV}{\ensuremath{{\mathcal{V}}}}
\newcommand{\cW}{\ensuremath{{\mathcal{W}}}}

\title{Maximal green sequences for cluster-tilted algebras of finite representation type}

\author{Kiyoshi Igusa}
\address{Department of Mathematics, Brandeis University, Waltham, MA 02454}\email{igusa@brandeis.edu}

\address{Department of Mathematics, Northeastern University, Boston, MA 02115}\email{g.todorov@northeastern.edu}



\begin{document}

\begin{abstract} We show that, for any cluster-tilted algebra of finite representation type over an algebraically closed field, the following three definitions of a maximal green sequence are equivalent: (1) the usual definition in terms of Fomin-Zelevinsky mutation of the extended exchange matrix, (2) a forward hom-orthogonal sequence of Schurian modules, (3) the sequence of wall crossings of a generic green path. Together with \cite{PartI}, this completes the foundational work needed to support the author's work with P.J. Apruzzese \cite{AI}, namely, to determine all lengths of all maximal green sequences for all quivers whose underlying graph is an oriented or unoriented cycle and to determine which are ``linear''.

In an Appendix, written jointly with G. Todorov, we give a conjectural description of maximal green sequences of maximum length for any cluster-tilted algebra of finite representation type. 
\end{abstract}

\maketitle


\section{Introduction}

This paper is the second of three papers on the problem of ``linearity'' of stability conditions, namely: Is the longest maximal green sequence for an algebra equivalent to one given by a ``central charge''? Although we do not address this question in this paper, we explain the motivation behind the series of papers of which this is a part. The question originates from a conjecture by Reineke \cite{R} which asks if, for a Dynkin quiver, there is a ``slope function'' (a special case of a central charge) making all modules stable. Reineke wanted such a result because, when it holds, his formulas would then give an explicit description of a PBW basis for the quantum group $\cU_v(\frak n^+)$ for the Dynkin quiver. Yu Qiu partially answered this question by constructing a central charge for at least one orientation of each Dynkin quiver making all modules stable \cite{Q}.

The linearity problem is explained in the first paper \cite{PartI} and solved in the third paper of this series \cite{AI} in three cases: (1) a hereditary algebra of affine type $\tilde A_{n-1}$, (2) type $A_n$ any orientation and (3) the $n$-cycle quiver modulo $rad^{n-1}$. The last example is well-known to be cluster-tilted of type $D_n$ and is thus among the ones discussed in this paper.

In the third paper \cite{AI}, we use the description of a maximal green sequence in terms of ``wall crossing'' sequences. The main purpose of the first two papers is to prove, in the three cases considered in \cite{AI}, that the wall crossing description is equivalent to the usual definition of a maximal green sequence in terms of Fomin-Zelevinsky mutation of a skew-symmetrizable matrix called the ``exchange matrix'' \cite{FZ}. This definition is reviewed in the example below. By ``finite type'' we mean the associated cluster algebra has finite type. By \cite{FZ2} this is equivalent to the exchange matrix being mutation equivalent to an acyclic exchange matrix of Dynkin type. We use the cluster category approach \cite{BMRRT}. To any acyclic exchange matrix we can associate an hereditary algebra $H$ and the clusters of the associated cluster algebra are in bijection with the cluster-tilting objects of the cluster category $\cC_H$ of $H$. The endomorphism ring of a cluster-tilting object is called a ``cluster-tilted algebra''. These were introduced in \cite{BMR}. (See the Appendix below for more details, including definitions, for ``tilted algebras'' and ``cluster-tilted algebras''.) A cluster-tilted algebra has finite representation type if and only if it comes from a cluster category of finite type \cite{BMR2} (see also \cite{KZ}). This happens if and only if the hereditary algebra has finite representation type, equivalently it is of Dynkin type. By \cite{FZ2} this happens if and only if the cluster algebra has finite type. By the main results of \cite{PartI}, this implies that the maximal green sequences of cluster algebras of finite type are in bijection with maximal green sequences for a corresponding (not unique) cluster-tilting algebra which are given by sequences of indecomposable modules. We use representation theory to study these sequences.

In this paper we restrict to the case of cluster algebras of finite type coming from skew-symmetric matrices. To each such algebra there is an associated quiver with potential \cite{DWZ}. As explained in the previous paragraph, there is a corresponding cluster-tilted algebra of finite representation type. In the skew-symmetric case we can choose the cluster-tilted algebra to be an algebra over an algebraically closed field. In this case, the cluster-tilted algebra is well-known to be isomorphic to the Jacobian algebra $\Lambda=J(Q,W)$ of a quiver with potential of finite representation type \cite{BIRSm}. (See the example given below.) The main theorem of this paper is the equivalence between the following three definitions of a maximal green sequence for such an algebra $\Lambda$.

\begin{enumerate}
\item The usual definition (mutation of the extended exchange matrix)
\item Complete forward hom-orthogonal sequence of Schurian modules (Definition \ref{def of FHO}). Recall that a module is \emph{Schurian} if its endomorphism ring is a division algebra.
\item Wall crossing sequence of a green path in the cluster complex (Definition \ref{def: green path}).
\end{enumerate}

The precise statement is as follows.

\begin{thm}\label{main thm}
Let $\Lambda=J(Q,W)$ be the Jacobian algebra of a finite type quiver $Q$ with nondegenerate potential $W$ over any field. Let $\beta_1,\cdots,\beta_m$ be a finite sequence of elements of $\NN^n$ where $n$ is the number of vertices of $Q$. Then the following are equivalent.
\begin{enumerate}
\item $\beta_1,\cdots,\beta_m$ are the $c$-vectors of a maximal green sequence for $Q$.
\item There exist Schurian left $\Lambda$-modules $M_1,\cdots,M_m$ with dimension vectors $\undim M_i=\beta_i$ so that
\begin{enumerate}
\item $\Hom_\Lambda(M_i,M_j)=0$ for all $i<j$.
\item No other module can be inserted into this sequence preserving (a).
\end{enumerate}
\item There exists a generic green path $\gamma$ crossing a finite number of stability walls $D(M_1)$, ... , $D(M_m)$ so that $\undim M_i=\beta_i$. 
\end{enumerate}
\end{thm}

In Section 2 we prove the equivalence $(1)\Leftrightarrow (2)$ and, in Section 3, we prove $(2)\Leftrightarrow (3)$.
In \cite{PartI} we proved the equivalences $(1)\Leftrightarrow (3)\Leftrightarrow (2)$ for $\Lambda$ any finite dimensional hereditary algebra over any field. The equivalence $(1)\Leftrightarrow (3)$ for cluster-tilted algebras of type $D_n$ and tame hereditary algebras of type $\tilde A_{n-1}$ is what is needed for the last paper \cite{AI}.

In an Appendix, written jointly with Gordana Todorov, we use this equivalence to describe upper and lower bounds on the maximal length of a maximal green sequence for $\Lambda$ and we conjecture that the lower bound is sharp. When this Conjecture (\ref{conjectural description of maximal MGSs}) holds, it gives a representation theoretic description of maximal green sequences of maximal length.

As an example of the conjecture, consider the following quiver.
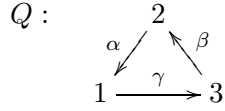
\[
\xymatrixrowsep{15pt}\xymatrixcolsep{10pt}
\xymatrix{
Q:&& 2\ar[dl]_\alpha\\
&1\ar[rr]^\gamma &&3\ar[lu]_\beta
	}
\]
with potential $W=\alpha\beta\gamma$. Potentials are linear combinations of oriented cycles in the quiver and are only well-defined up to cyclic permutation. The \emph{Jacobian algebra} $J(Q,W)$ is defined to be the path algebra $KQ$ of $Q$ over a fixed field $K$ modulo the partial derivatives of $W$ with respect to each arrow \cite{DWZ}. In this case these relations are:
\[
	\partial_\alpha W:=\beta\gamma=0,\quad \partial_\beta W:=\gamma\alpha=0,\quad \partial_\gamma W:=\alpha\beta=0.
\]
In other words, $\Lambda=J(Q,W)=KQ/rad^2KQ$. 

As we explain in the Appendix, there are three ``tilted algebras'' associated to $\Lambda$. They are given by ``cutting'' the quiver, i.e., by deleting one of the three edges of $Q$. Any two of the quivers obtained by cutting are isomorphic. So, we delete edge $\gamma$. We are left with the following quiver:
\[
	Q^\delta:\quad 1\xleftarrow\alpha 2\xleftarrow\beta 3
\]
with one relation $\alpha\beta=0$ since this is the only relation of $Q$ supported on the subquiver $Q^\delta$. The resulting algebra $C=Q^\delta/(\alpha\beta)$ is a tilted algebra of type $A_3$ with 5 indecomposable modules with dimension vectors $(0,0,1), (0,1,1), (0,1,0), (1,1,0), (1,0,0)$. Conjecture \ref{conjectural description of maximal MGSs} then states that a longest maximal green sequence for $Q$ has these vectors as the $c$-vectors of the mutations as we now explain and thus has length 5.

Maximal green sequences (MGS) are defined only in terms of the quiver $Q$ as follows. The \emph{exchange matrix} $B$ of $Q$ is defined to be the skew-symmetric matrix with $ij$th entry $b_{ij}$ equal to the number of arrows $i\to j$ in the quiver $Q$ minus the number of arrows $j\to i$. 
\[
	B=\mat{0 & -1 & 1\\
	1 & 0 & -1\\
	-1 & 1 & 0}
\]
The \emph{extended exchange matrix} $\tilde B$ is given by putting the $3\times 3$ identity matrix $I_3$ below $B$. By Theorem \ref{thm: MSG for An} there is a MGS of length 5 which mutates at the ``$c$-vectors'' which are the dimension vectors of the $C$-modules as follows.
\[
	\tilde B=\mat{0 & -1 & 1\\
	1 & 0 & -1\\
	-1 & 1 & 0\\
	\hline 
	1 & 0 & \color{green}0\\
	0 & 1 & \color{green}0\\
	0 & 0 & \color{green}1} 
	\xrightarrow{\mu_3}
	 \mat{0 & 0 & -1\\
	0 & 0 & 1\\
	1 & -1 & 0\\
	\hline 
	1 & \color{green}0 & 0\\
	0 & \color{green}1 & 0\\
	0 & \color{green}1 & -1} 
	\xrightarrow{\mu_2} 
	 \mat{0 & 0 & -1\\
	0 & 0 & -1\\
	1 & 1 & 0\\
	\hline 
	1 & 0 & \color{green}0\\
	0 & -1 & \color{green}1\\
	0 & -1 & \color{green}0} 
	\xrightarrow{\mu_3} 
	\mat{0 & 0 & 1\\
	0 & 0 & 1\\
	-1 & -1 & 0\\
	\hline 
	\color{green}1 & 0 & 0\\
	\color{green}1 & 0 & -1\\
	\color{green}0 & -1 & 0} 
	\]
	\[
	\xrightarrow{\mu_1} 
	\mat{0 & 0 & -1\\
	0 & 0 & 1\\
	1 & -1 & 0\\
	\hline 
	-1 & 0 & \color{green}1\\
	-1 & 0 & \color{green}0\\
	0 & -1 & \color{green}0}	
	\xrightarrow{\mu_3} 
	\mat{0 & -1 & 1\\
	1 & 0 & -1\\
	-1 & 1 & 0\\
	\hline 
	0 & 0 & -1\\
	-1 & 0& 0\\
	0 & -1 & 0}
\]
The \emph{$c$-vectors} are the columns of the bottom half of these exchange matrices. The Fomin-Zelevinsky mutation $\mu_k$ is given by applying the following general rule. We add $b_{ik}|b_{kj}|$ to $b_{ij}$ whenever $b_{ik}$ and $b_{kj}$ have the same sign, then we change the signs of all $b_{ik},b_{kj}$ \cite{FZ}. The mutation $\mu_k$ is \emph{green} if the $k$th $c$-vector is positive, i.e., all its entries are nonnegative. The last matrix has no positive $c$-vectors. So, the green mutation sequence must stop. This is therefore a \emph{maximal green sequence}. The standard notation for this mutation sequence is $(3,2,3,1,3)$. However, we prefer to label the sequence with its $c$-vectors because of their representation theoretic interpretation.

The $c$-vectors of this MGS are the dimension vectors of the $C$-modules in the order $M_1,\cdots,M_5$ so that $\Hom_{C}(M_i,M_j)=\Hom_{\Lambda}(M_i,M_j)=0$ for $i<j$. The sixth $\Lambda$-module $M_6$ with $\undim M_6=(1,0,1)$ cannot be inserted into this sequence since $\Hom_\Lambda(M_1,M_6)\neq 0$ and $\Hom_\Lambda(M_6,M_5)\neq 0$. So, the sequence $M_1,\cdots,M_5$ is a \emph{complete forward hom-orthogonal sequence}.

\section{Forward hom-orthogonal sequences}

We will show that, for any Jacobian algebra $\Lambda=J(Q,W)$ of finite representation type, there is a $1$-$1$ correspondence between maximal green sequences and complete forward hom-orthogonal sequences of Schurian modules, also known as ``bricks''. We prove this by induction on the length of a green sequence. The key step (Lemma \ref{lem: inductive step A=B for Q finite type}) is known to hold for any cluster-tilted algebra by \cite{BIRSm}.
But, we go through the details for the benefit of our students.
\subsection{Definition of forward hom-orthogonal sequence}

The results of this subsection hold for $\Lambda$ any finite dimensional algebra over a field.

In the following definition we use the notation $\cF(M):=M^\perp$ for the full subcategory of $\Lambda\text-mod$ of all $Y$ so that $\Hom_\Lambda(M,Y)=0$ and $\cG(M):=\,^\perp\cF(M)$ is the full subcategory of $\Lambda\text-mod$ of all $X$ so that $\Hom_\Lambda(X,Y)=0$ for all $Y\in\cF(M)$. Then $\cF(M)=\cG(M)^\perp$. So, $(\cG(M),\cF(M))$ is a \emph{torsion pair} in $\Lambda\text-mod$ where $\cG(M)$ (resp. $\cF(M)$) is the \emph{torsion class} (resp., \emph{torsion-free class}). (See, e.g., \cite{AR}.) 

The basic properties of the {torsion class} $\cG(M)$ are that $\cG(M)$ is closed under extensions (in particular, direct sums) and also closed under taking quotients. For example, any quotient module $N$ of $M$ lies in $\cG(M)$ and therefore, $\cG(N)\subset\cG(M)$ and $\cF(N)=N^\perp\supset M^\perp=\cF(M)$. We also need the fact that, for any $\Lambda$-module $X$, there is a \emph{canonical short exact sequence}
\[
	0\to rX\to X\to tX\to 0
\]
where $rX\in \cG(M)$ and $tX\in \cF(M)$. The canonical sequence is easy to construct: $rX$ is simply the sum of all submodules of $X$ which are objects of $\cG(M)$. Since $\cG(M)$ is closed under direct sums and quotients, $rX$ is an object in $\cG(M)$.


\begin{prop}\label{prop: HN stratification with 3 strata}
Suppose $\Hom_\Lambda(X,Y)=0$ and $\cC=X^\perp\cap\,^\perp Y$. 
\begin{enumerate}
\item If $\cC\neq0$ then $\cC$ contains a Schurian object.
\item $\cG(X)=\,^\perp \cC\cap\,^\perp Y$.
\end{enumerate}
\end{prop}

\begin{proof} (1) Let $M$ be a nonzero object of $\cC$ of minimal length. Then we claim that $M$ is Schurian. If not, there would be an endomorphism $f:M\to M$ whose image $f(M)$ is nonzero and of smaller length than $M$. But $f(M)\in X^\perp$, being a submodule of $M$ and $f(M)\in \,^\perp Y$ since it is a quotient object of $M$. Therefore, $f(M)$ lies in $\cC$ which contradicts the minimality of $M$.

(2) $\cG(X)\subset \,^\perp Y$ since $Y\in \cF(X)$. $\cG(X)\subset \,^\perp \cC$ since $\cC\subset  X^\perp=\cF(X)$. Therefore, $\cG(X)\subset \,^\perp\cC\cap \,^\perp Y$. Conversely, let $M\in \,^\perp \cC\cap\,^\perp Y$. Then, in the canonical short exact sequence $0\to rM\to M\to tM\to 0$ for the torsion pair $(\cG(X),\cF(X))$, $tM\in  \,^\perp \cC\cap\,^\perp Y$ since $tM$ is a quotient of $M\in \,^\perp \cC\cap\,^\perp Y$. Also, $tM\in \cF(X)= X^\perp$ by definition of the canonical sequence. So, $tM\in X^\perp\cap \,^\perp\cC\cap\,^\perp Y=\cC\cap \,^\perp \cC=0$. So, $M=rM\in\cG(X)$.
\end{proof}

\begin{defn}\label{def of FHO}
A \emph{forward hom-orthogonal (FHO) sequence} in $\Lambda\text-mod$ is a finite sequence of Schurian modules $M_1, \cdots, M_m$ so that
\begin{enumerate}
\item $\Hom_\Lambda(M_i,M_j)=0$ for all $1\le i<j\le m$.
\item The sequence is maximal in $\cG(M)$ where $M=M_1\oplus \cdots\oplus M_m$. By \emph{maximal} we mean that no other Schurian object in $\cG(M)$ can be inserted into the sequence $M_1,\cdots,M_m$ preserving property (1).
\end{enumerate}
We say that a forward hom-orthogonal sequence is \emph{complete} if $\cG(M)=\Lambda\text-mod$ or, equivalently, $\cF(M)=0$ where we generally let $M$ denote $M_1\oplus\cdots\oplus M_m$ when it is clear from the context which forward hom-orthogonal sequence is being considered. If the sequence $(M_i)$ satisfies only (1) we call it \emph{weakly forward hom-orthogonal}. Note that, by Proposition \ref{prop: HN stratification with 3 strata}, property (2) is equivalent to the condition that $\cG(M)\cap(M_1\oplus \cdots \oplus M_k)^\perp\cap\,^\perp(M_{k+1}\oplus\cdots\oplus M_m)=0$ for all $k$ since, if this intersection were nonzero, by \ref{prop: HN stratification with 3 strata}(1) it would contain a Schurian object which could be inserted between $M_k$ and $M_{k+1}$ in the FHO sequence.
\end{defn}

As a consequence of Proposition \ref{prop: HN stratification with 3 strata} we have the following.

\begin{cor}\label{cor: max weakly forward hom-orth is complete}
Let $M_1,\cdots,M_m$ be a weakly FHO sequence of Schurian $\Lambda$-modules which is \emph{maximal} in the sense that it is not a proper subsequence of any other weakly FHO sequence. Then $M_1,\cdots,M_m$ is complete.
\end{cor}

\begin{proof}
It follows from Definition \ref{def of FHO} that $M_1,\cdots,M_m$ is FHO in $\cG(M)$. It remains to show that $\cG(M)=\Lambda\text-mod$. Suppose not. Then, $\cF(M)\neq0$ and, by Proposition \ref{prop: HN stratification with 3 strata} with $Y=0$ and $X=M$, there would exist a Schurian module in $\cF(M)=M^\perp$. This Schurian module could be added after $M_m$ contradicting its maximality.
\end{proof}

\begin{prop}\label{prop:subsequence of hom orth is hom orth}
If $M_1,\cdots,M_m$ is a FHO sequence in $\Lambda\text-mod$ then so is $M_1,\cdots,M_k$ for any $k<m$.
\end{prop}

\begin{proof}
Suppose not. Then there is a Schurian object $X$ in $\cG(M_1\oplus \cdots\oplus M_k)$ which can be inserted into the sequence $M_1,\cdots,M_k$. But $X\in \,^\perp M_j$ for $k<j\le m$ since $X\in \cG(M_1\oplus\cdots\oplus M_k)\subset \cG(M_1\oplus\cdots\oplus M_m)$. So, $X$ can also be inserted into the original sequence $M_1,\cdots,M_m$ giving a contradiction.
\end{proof}

\begin{lem}\label{lem:when Sk is not one of the Mi}
Let $M_1,\cdots,M_m$ be a FHO sequence in $\Lambda\text-mod$ so that no $M_i$ is isomorphic to the simple module $S_j$. Then $\Hom_\Lambda(M_i,S_j)=0$ for all $i$.
\end{lem}

\begin{proof}
Suppose not and let $i$ be minimal so that $\Hom_\Lambda(M_i,S_j)\neq0$. Since $S_j$ is simple, any nonzero morphism $M_i\to S_j$ must be an epimorphism. So, $S_j\in\cG(M)$, where $M=\bigoplus_iM_i$. Also, $\Hom_\Lambda(S_j,M_k)=0$ for all $i\le k\le m$. So, $S_j$ can be inserted into the FHO sequence to the left of $M_i$ contradicting its maximality.
\end{proof}

This Lemma implies that, if a simple module is missing from a FHO sequence, it can be added after the last module.

\begin{prop}
Each complete FHO sequence in $\Lambda\text-mod$ contains every simple $\Lambda$-module.\qed
\end{prop}

\begin{cor}\label{cor: every comple fhos ends in a simple}
If $M_1,\cdots,M_m$ is a complete FHO sequence then $M_m$ is simple.
\end{cor}

\begin{proof}
If $M_m$ were not simple, all simple $\Lambda$-modules would come before $M_m$ in the sequence. Then, $\Hom_\Lambda(S_k,M_m)=0$ for all $k$ which is impossible.
\end{proof}

\begin{lem}\label{lem:FHO of length one is simple}
Let $M$ be a Schurian module. Then $M=M_1$ is a FHO sequence of length $1$ if and only if $M$ is simple.
\end{lem}

\begin{proof}
Let $M=S_i$ be the $i$th simple module. Then $\cF(S_i)$ contains the injective modules $I_j$ for all $j\neq i$. So, $\cG(S_i)$ consists only of iterated self extensions of $S_i$. So, $S_i$ is the only Schurian object in $\cG(S_i)$. So, it is a FHO sequence of length 1.

Conversely, let $M=M_1$ be a FHO sequence of length 1. If $M$ is not simple then, by Lemma \ref{lem:when Sk is not one of the Mi}, $\Hom_\Lambda(M,S_j)=0$ for all $S_j$ which is impossible.
\end{proof}

Proposition \ref{prop:subsequence of hom orth is hom orth}, Lemma \ref{lem:FHO of length one is simple} and Proposition \ref{cor: every comple fhos ends in a simple} imply that every FHO sequence starts with a simple module. This leads to the following very useful result.

\begin{prop}\label{prop: complete iff M1 is simple and M2, etc is maximal}
Let $M_1,\cdots,M_m$ be a sequence of Schurian $\Lambda$-modules. Then the following are equivalent.

(1) $M_1,\cdots,M_m$ is a complete FHO sequence in $\Lambda\text-mod$.

(2) $M_1$ is simple and $M_2,\cdots,M_m$ is a maximal weakly FHO sequence in $M_1^\perp$.

(3) $M_m$ is simple and $M_1,\cdots,M_{m-1}$ is a maximal weakly FHO sequence in $\,^\perp M_m$.
\end{prop}

\begin{proof} Since every complete FHO sequence starts and ends with a simple module, (1) implies (2) and (3).

For the converse, we observe in both cases (2) and (3), every simple module must occur in the sequence. Suppose not. Let $S_k$ be a simple module not isormorphic to any $M_i$. In Case (3), $M_1,\cdots,M_{m-1}$ is FHO by definition since $\,^\perp M_m$ is a torsion class. Therefore, by Lemma \ref{lem:when Sk is not one of the Mi}, $\Hom_\Lambda(M_i,S_k)=0$ for all $i<m$. Also, $\Hom_\Lambda(S_k,M_m)=0$ since $S_k,M_m$ are nonisomorphic simple modules. So, $S_k$ can be added to the sequence $M_1,\cdots,M_{m-1}$ contradicting its maximality as a sequence in $\,^\perp M_m$. In case (2), let $j\ge1$ be maximal so that $\Hom_\Lambda(M_i,S_k)=0$ for $1\le i\le j$. Then, either $j<m$ in which case $\Hom_\Lambda(M_{j+1},S_k)\neq0$ or $j=m$. In either case, $S_k$ can be inserted into the sequence $M_1,\cdots,M_m$ after $M_j$, where $j\ge1$, contradicting the maximality of $M_2,\cdots,M_m$ in $M_1^\perp$. 

This implies that $M_1,\cdots,M_m$ is maximal weakly FHO and thus complete (by Corollary \ref{cor: max weakly forward hom-orth is complete}). Indeed, in Case (2), no module can be inserted before $M_1$ since every module maps to at least one simple module. No Schurian module can be inserted after $M_1$ since that would increase the length of the sequence $M_2,\cdots,M_m$ in $M_1^\perp$ contradicting its maximality. Case (3) is similar. Therefore, the three conditions are equivalent.
\end{proof}

\subsection{Rotation of forward hom-orthogonal sequences} We assume for the rest of the paper that $Q$ is a quiver without loops of oriented 2-cycles which is of finite type in the sense of Fomin and Zelevinsky \cite{FZ}. Let $W$ be a nondegenerate potential for $Q$ in the sense of \cite{DWZ}. Let $\Lambda=J(Q,W)$ be the Jacobian algebra of the quiver with potential $(Q,W)$. Mutation of quivers with potential will be reviewed below (before the proof of Lemma \ref{lem: inductive step A=B for Q finite type}).

We will show, in Theorem \ref{thm: inductive thm implying main thm}, that sequences of $m$ green mutations (i.e., mutations on positive $c$-vectors) of the quiver $Q$ are in bijection with FHO sequences of length $m$ in $J(Q,W)\text-mod$. This will immediately imply the bijection (Corollary \ref{cor: main thm of section 2}):
\[
	\left\{\begin{array}{cc}
	\text{maximal green sequences}\\
	\text{ for $Q$}
	\end{array}
	\right\}
\cong 
	\left\{\begin{array}{cc}
	\text{complete FHO sequences}\\
	\text{ for $J(Q,W)$}
	\end{array}
	\right\}.
\]

We will prove the $m$ mutation statement by induction on $m$ using Lemma \ref{lem: inductive step A=B for Q finite type} below which, by Corollary \ref{cor: rotation lemma}, is an analogue of the Rotation Lemma \cite{BHIT} for complete FHO sequences. Before this we review the mutation process for quivers with potential: $(Q',W')=\mu_k(Q,W)$ assuming that $Q$ is a quiver of finite type with nondegenerate potential $W$. We begin with the observation that, since $Q$ has finite type, it cannot contain the subquiver 
\[
	\tilde A_2:\xymatrix{ i\ar@/^1pc/[rr]\ar[r] &k\ar[r] & j}
\]
since, if it did, mutation at $k$ would produce a double arrow from $i$ to $j$.

Let $I$ be the set of all vertices $i$ for which there is an arrow $\alpha_i:i\to k$ in $Q$. Let $J$ be the set of all vertices $j$ for which there is an arrow $\beta_j:k\to j$ in $Q$. Since $Q$ has finite type, $Q$ has no arrows between elements of $I$ and no arrows between elements of $J$. Otherwise $Q$ would contain a subquiver of type $\tilde A_2$. Additionally, for each $(i, j) \in  I \times J$ there are no arrows $i \rightarrow j$ and there is at most one arrow $j \rightarrow i$.

Let $P$ be the set of all pairs $(i,j)\in I\times J$ for which $Q$ has an arrow $\gamma_{ij}:j\to i$. Let $P'$ be the complement of $P$ in $I\times J$. The mutated quiver $Q'=\mu_kQ$ is obtained from $Q$ by reversing the orientations of the arrows $\alpha_i,\beta_j$ to produce new arrows $\alpha_i^\ast:k\to i$, $\beta_j^\ast:j\to k$, adding an arrow $\gamma_{i'j'}^\ast:i'\to j'$ for each $(i',j')\in P'$ and deleting the arrows $\gamma_{ij}:j\to i$ for all $(i,j)\in P$. This is indicated as follows.
\[
\xymatrix{
Q: & j\ar@/^1pc/[rr]^{\gamma_{ij}} &k\ar[l]^{\beta_j} & i\ar[l]^{\alpha_i}
&& Q'=\mu_k Q: & j'\ar[r]_{\beta_{j'}^\ast}
&k \ar[r]_{\alpha_{i'}^\ast} & i'\ar@/_1pc/[ll]_{\gamma_{i'j'}^\ast}
	}
\]

The potential $W$ is a linear combination of oriented cycles in $Q$. We begin with the general form of such a potential then simplify it by a change of coordinates. Collecting together all terms involving $\gamma_{ij}$ or $\beta_j\alpha_i$ for $(i,j)\in P$, we have:
\[
	W=\sum_{(i,j)\in P}a_{ij}\gamma_{ij}\beta_j\alpha_i+\sum_{(i,j)\in P}A_{ij}\beta_j\alpha_i+\sum_{(i,j)\in P}\gamma_{ij}B_{ij}+ \text{ other terms}
\]
where $a_{ij}$ are nonzero scalars, $A_{ij},B_{ij}$ are linear combinations of path $j\to i$ and $i\to j$ of length at least 2. We claim that, by changing the choice of $\gamma_{ij}$ we may eliminate the terms $A_{ij}$. The new $\gamma_{ij}$ will be equal to $a_{ij}\gamma_{ij}$ plus a converging, possibly infinite sum of paths $j\to i$.

The first step in this simplification process is given by changing the choice of $\gamma_{ij}$ to $\gamma_{ij}':=a_{ij}\gamma_{ij}+A_{ij}$. Equivalently, we make the substitution 
\[
\gamma_{ij}=a_{ij}^{-1}\gamma_{ij}'-a_{ij}^{-1}A_{ij}.
\]
Then we obtain:
\[
	W=\sum_{(i,j)\in P}\gamma_{ij}'\beta_j\alpha_i+\sum_{(i,j)\in P}A_{ij}'\beta_j\alpha_i+\sum_{(i,j)\in P}\gamma_{ij}'B_{ij}'+ \text{ other terms}.
\]
The old terms $A_{ij}\beta_j\alpha_i$ are cancelled. However, new term arise from $\gamma_{ij}B_{ij}=a_{ij}^{-1}\gamma_{ij}'B_{ij}-a_{ij}^{-1}A_{ij}B_{ij}$. The term $a_{ij}^{-1}A_{ij}B_{ij}$ might give a term $A_{ij}'\beta_j\alpha_i$ since $\beta_j\alpha_i$ might be a subword of a term in $A_{ij}$. To fix this we let $\gamma_{ij}'':=\gamma_{ij}'+A_{ij}'$ and substitute again:
\[
	\gamma_{ij}'=\gamma_{ij}''-A_{ij}'.
\]
This process converges because the minimum length of the paths in the sum $A_{ij}$ is going to infinity. To see this, define the ``adjusted length'' of each path in $A_{ij}$ to be its length minus the number of two letter subwords of the form $\beta_j\alpha_i$. When a cycle in the sum $A_{ij}B_{ij}$ becomes a cycle in $A_{ij}'\beta_j\alpha_i$, the length of the path in $A_{ij}'$ is at least equal to the length of the corresponding path in $A_{ij}$ since $B_{ij}$ is a linear combination of paths of length $\ge2$. However, the adjusted length strictly increases since the subword $\beta_j\alpha_i$ has been removed from $A_{ij}$. Therefore, the process converges. So, we may assume $a_{ij}=1$ and $A_{ij}=0$ in the expression for $W$ and we have:
\[
	W=\sum_{(i,j)\in P}\gamma_{ij}\beta_j\alpha_i+F
\]
where $F$ is a linear combination of oriented cycles in $Q$ none of which contain the subword $\beta_j\alpha_i$ for any $(i,j)\in P$. We may also assume that each of these oriented cycles, written as a path, do not start at vertex $k$. Then, if $\alpha_i$ occurs as a letter in one of these paths, the next letter must be $\beta_{j}$ where $(i,j)\in P'$. This implies that, for each $i\in I$, we have
\[
	\partial_{\alpha_i}F=\sum_{j:(i,j)\in P'} F_{ij}\beta_j
\]
where $F_{ij}$ is a linear combination of paths from $j$ to $i$. For each occurrence of the word $\beta_j\alpha_i$ in the terms in $F$ we get one term in $F_{ij}$. With this description we also see that
\[
	\partial_{\beta_j}F=\sum_{i:(i,j)\in P'} \alpha_iF_{ij}.
\]
For each $(i,j)\in P$ let
\[
	G_{ij}:=\partial_{\gamma_{ij}}F
\]

The Jacobian algebra $J(Q,W)$ is the path algebra of $Q$ modulo the following relations.
\[
	(\forall (i,j)\in P)\quad \partial_{\gamma_{ij}}W=0:\qquad\qquad \beta_j\alpha_i=-G_{ij}\qquad\quad 
\]
\begin{equation}\label{partial alpha-i W=0}
	(\forall i\in I)\quad \partial_{\alpha_i}W=0:\qquad \sum_{j:(i,j)\in P} \gamma_{ij}\beta_j+\sum_{j':(i,j')\in P'} F_{ij'}\beta_{j'} =0
\end{equation}
\[
	(\forall j\in J)\quad \partial_{\beta_j}W=0:\qquad \sum_{i: (i,j)\in P} \alpha_i\gamma_{ij}+\sum_{i': (i',j)\in P'} \alpha_{i'}F_{i'j} =0
\]
And, for all other arrows $\delta$, the equation $\partial_\delta W=0$ is just $W_\delta:=\partial_\delta F=0$.

We need the following notation. Let $X$ be a linear combination of paths or oriented cycles in $Q$ which do not start or end at vertex $k$ and which do not contain the subword $\beta_j\alpha_i$ for any $(i,j)\in P$, for example, $X=F,F_{ij},G_{ij}$. We denote by $\widetilde X^\ast$ the corresponding linear combination of paths or cycles in $Q'$ given by replacing each occurrence of the letter $\gamma_{ij}$ in $X$, for $(i,j)\in P$, with $\alpha_i^\ast\beta_j^\ast$ and each occurrence of the pair of letters $\beta_{j'}\alpha_{i'}$, for $(i',j')\in P'$, with $\gamma_{i'j'}^\ast$. 

We need formulas for commuting cyclic derivatives with the operation $\widetilde{(\ )}^\ast$. The best explanation might be by example: Given an oriented cycle $X=(\gamma_{ij})ab(\beta_{j'}\alpha_{i'})de(\gamma_{ij'})gh$ we have $\widetilde X^\ast=
(\alpha_i^\ast\beta_j^\ast)  ab( \gamma_{i'j'}^\ast )de(\alpha_i^\ast\beta_{j'}^\ast )gh$. Then 
\[
\partial_{\alpha_i^\ast} \widetilde X^\ast=\beta_j^\ast  ab( \gamma_{i'j'}^\ast )de(\alpha_i^\ast\beta_{j'}^\ast )gh+
\beta_{j'}^\ast gh (\alpha_i^\ast\beta_j^\ast)  ab( \gamma_{i'j'}^\ast )de
= \beta_j^\ast \widetilde{\partial_{\gamma_{ij}}X}^\ast+\beta_{j'}^\ast \widetilde{\partial_{\gamma_{ij'}}X}^\ast
\]
following the general formula:
\[
	\partial_{\alpha_i^\ast} \widetilde X^\ast=\sum_{j':(i,j')\in P'}\beta_{j'}^\ast \widetilde{\partial_{\gamma_{ij'}}X}^\ast .
\]
Using this formula, we obtain the following calculations.
\begin{equation}\label{eq: derivative of tilde F}
	\forall i\in I:\qquad \partial_{\alpha_i^\ast}\widetilde F^\ast=\sum_{j:(i,j)\in P}\beta_{j}^\ast \widetilde{\partial_{\gamma_{ij}}F}^\ast
	=\sum_{j:(i,j)\in P}\beta_{j}^\ast \widetilde G_{ij}^\ast
\end{equation}
\[
	\forall j\in J:\qquad \partial_{\beta_j^\ast}\widetilde F^\ast=\sum_{i:(i,j)\in P} \widetilde G_{ij}^\ast\alpha_{i}^\ast
\]
\[
	\forall (i',j')\in P': \qquad \partial_{\gamma_{ij}^\ast}\widetilde F^\ast=\widetilde F_{ij}^\ast
\]
All other arrows $\delta$ in $Q'$ are also arrows in $Q$ and the operations $\partial_\delta$ and $\widetilde{(\ )}^\ast$ commute:
\[
	\partial_\delta \widetilde F^\ast=\widetilde{\partial_\delta F}^\ast.
\]

\begin{prop}\label{prop: formula for mutation of W}  
The mutation $\mu_k(Q,W)$ of this quiver with potential in direction $k$ is equivalent to $(Q',W')$ where $Q'$ is defined above and
\[
	W'=-\sum_{(i',j')\in P'} \alpha_{i'}^\ast\beta_{j'}^\ast\gamma_{i'j'}^\ast+\widetilde F^\ast
\]
where the notation $\widetilde F^\ast$ is defined above. The Jacobian algebra $J(Q',W')$ is the path algebra of $Q'$ modulo the relations given as follows.
\[
	(\forall i\in I)\quad \partial_{\alpha_i^\ast}W'=0:\qquad -\sum_{j':(i,j')\in P'} \beta_{j'}^\ast\gamma_{ij'}^\ast +
	\sum_{j:(i,j)\in P}\beta_{j}^\ast \widetilde G_{ij}^\ast=0
\]
\[
	(\forall j\in J)\quad \partial_{\beta_j^\ast}W'=0:\qquad -\sum_{i': (i',j)\in P'} \gamma_{i'j}^\ast\alpha_{i'}^\ast+\sum_{i:(i,j)\in P} \widetilde G_{ij}^\ast\alpha_{i}^\ast =0
\]
\[
	(\forall (i',j')\in P')\quad \partial_{\gamma_{i'j'}^\ast}W'=0:\qquad\qquad \alpha_{i'}^\ast\beta_{j'}^\ast=\widetilde F_{i'j'}^\ast\qquad\quad 
\]
And, for all other arrows $\delta$, the equation $\partial_\delta W'=0$ is just $\widetilde W_\delta^\ast=0$ where $W_\delta=\partial_\delta  F$.
\end{prop}

\begin{proof}
We follow the original definition of mutation of a quiver with potential as given in \cite{DWZ}. The first step is to replace the pair of letters $\beta_j\alpha_i$ with a new letter $\gamma_{ij}^\ast$ for all $(i,j)\in I\times J=P\coprod P'$. We denote the result of such a procedure with a tilde $\widetilde\,$. We get:
\[
	\widetilde W= \sum_{(i,j)\in P} \gamma_{ij}\gamma_{ij}^\ast + \widetilde F
\]
where $\widetilde F$ is a linear combination of oriented cycles of length at least 3. The next step is usually to add new terms $\alpha_i^\ast\beta_j^\ast\gamma_{ij}^\ast$ for all $(i,j)\in I\times J$. However, it is more convenient to subtract these terms to get:
\[
	\widetilde W^+=-\sum_{(i,j)\in I\times J} \alpha_i^\ast\beta_j^\ast\gamma_{ij}^\ast+ \sum_{(i,j)\in P} \gamma_{ij}\gamma_{ij}^\ast + \widetilde F
\]
This deviation from standard procedure is justified since the letters $\alpha_i^\ast$ occur only in this new sum. By replacing each $\alpha_i^\ast$ with its negative, we change the signs of these new terms.

The next step is to change the basis again by replacing $\gamma_{ij}$ with $\gamma_{ij}'+\alpha_i^\ast\beta_j^\ast$ for each $(i,j)\in P$. This eliminates those terms in the first sum corresponding to $(i,j)\in P$ to give:
\[
	\widetilde W'=-\sum_{(i,j)\in P'} \alpha_i^\ast\beta_j^\ast\gamma_{ij}^\ast+ \sum_{(i,j)\in P} \gamma_{ij}'\gamma_{ij}^\ast + \widetilde F'
\]
The term $\widetilde F'$ contains the term that we want $\widetilde F^\ast$ plus other unwanted terms involving $\gamma_{ij}'$ for $(i,j)\in P$. So, $\widetilde F'=\widetilde F^\ast+\sum \gamma_{ij}'Z_{ij}$ where $Z_{ij}$ is a linear combination of paths of length at least 2 from $i$ to $j$ for each $(i,j)\in P$. Then $\widetilde W'$ can be rewritten as:
\[
	\widetilde W'=-\sum_{(i,j)\in P'} \alpha_i^\ast\beta_j^\ast\gamma_{ij}^\ast+ \sum_{(i,j)\in P} \gamma_{ij}'\gamma_{ij}^\ast+\sum_{(i,j)\in P} \gamma_{ij}'Z_{ij} + \widetilde F^\ast.
\]
Since $\gamma_{ij}^\ast$ only occurs in the second sum, we can change coordinates, replacing $\gamma_{ij}^\ast$ with $\gamma_{ij}'^{\ast}-Z_{ij}$ to eliminate the third, unwanted sum. Then
\[
	\widetilde W'=-\sum_{(i,j)\in P'} \alpha_i^\ast\beta_j^\ast\gamma_{ij}^\ast+ \sum_{(i,j)\in P} \gamma_{ij}'\gamma_{ij}'^\ast+ \widetilde F^\ast.
\]
Now the letters $\gamma_{ij}'$ and $\gamma_{ij}'^\ast$ occur only in the second sum which is a sum of 2-cycles. The final step is to ``reduce'' this expression by deleting these isolated 2-cycles to give
\[
	 W'=-\sum_{(i,j)\in P'} \alpha_i^\ast\beta_j^\ast\gamma_{ij}^\ast+  \widetilde F^\ast
\]
as claimed. Computation of the cyclic derivatives of $W'$ is given by Formulas \ref{eq: derivative of tilde F}.
\end{proof}


The following lemma is proved in much greater generality in \cite[Sec 7]{BIRSm} where the statement is demonstrated for any cluster-tilted algebra of the form $\Lambda=J(Q,W)$. The statement uses an involution $\varphi_k$ on $\ZZ^n$ defined as follows.

For any $x\in \ZZ^n$ and any $k\in Q_0$, let $\varphi_k(x)=y\in \ZZ^n$ be given by $y_i=x_i$ for $i\neq k$ and
\begin{equation}\label{eq: y=Xkx}
	y_k=-x_k+\sum_{k\to j} x_j
\end{equation}
where the sum is over all arrows $k\to j$ in $Q$.
Since $Q$ has no loops, $\varphi_k$ is an involution: $x=\varphi_k(y)$.

\begin{lem}\label{lem: inductive step A=B for Q finite type}
Let $S_k$ be a simple $\Lambda$-module for $\Lambda=J(Q,W)$ and let $\Lambda'$ be the Jacobian algebra of $(Q',W')=\mu_k(Q,W)$ where $W$ is a nondegenerate potential for $Q$. Suppose that $Q$ has finite type. Then there is an equivalence of full subcategories $\psi_k:S_k^\perp\cong \,^{\perp}S_k'$ where
\[
	S_k^\perp=\{X\in \Lambda\text-mod\,|\, \Hom_\Lambda(S_k,X)=0\}
\]
\[
	\,^{\perp}S_k'=\{Y\in \Lambda'\text-mod\,|\, \Hom_{\Lambda'}(Y,S_k')=0\}
\]
with $S_k,S_k'$ being the simple $\Lambda,\Lambda'$-modules at vertex $k$. Furthermore, the dimension vectors of $X$ and $Y=\psi_k(X)$ are related by
\[
	\undim \psi_k(X)=\varphi_k(\undim X).
\]
where $\varphi_k$ is the automorphism of $\ZZ^n$ given by \eqref{eq: y=Xkx} above.
\end{lem}

\begin{proof}

Any representation $X$ of $J(Q,W)$ is given by vector spaces $X_v$ for all vertices $v$ of $Q$ and linear maps $X_v\to X_w$ for arrows $v\to w$ satisfying the relations $\partial_\alpha W=0$. In particular, the arrows $k\to j$, for $j\in J$, induce linear maps $\beta_j:X_k\to X_j$. If $X\in S_k^\perp$, the sum of these linear maps will give a monomorphism $(\beta_j):X_k\to \bigoplus X_j$. Let $Y_k$ with maps $\beta_j^\ast:X_j\to Y_k$ be the cokernel. This gives a functorial short exact sequence:
\begin{equation}\label{functorial exact sequence}
	0\to X_k\xrightarrow{(\beta_j)}\bigoplus_{j\in J} X_j\xrightarrow{(\beta_j^\ast)}Y_k\to 0.
\end{equation}

Let $Y=\psi_k(X)$ be the representation of $\Lambda'=J(Q',W')$ given as follows. For each vertex $s\neq k$, let $Y_s=X_s$. $Y_k$ is given in \eqref{functorial exact sequence}. Then the exact sequence \eqref{functorial exact sequence} immediately implies
\[
	\undim Y=\varphi_k(\undim X)
\]
where $\varphi_k$ is the automorphism of $\ZZ^n$ given by \eqref{eq: y=Xkx} above. For each $j\in J$, the map $\beta_j^\ast:Y_j\to Y_k$ is given in \eqref{functorial exact sequence} since $Y_j=X_j$. Since these give an epimorphism $\bigoplus Y_j\onto Y_k$, we will have $\Hom_{\Lambda'}(Y,S_k)=0$. So $Y\in\,^{\perp}S_k'$ (once we show that $Y$ is a representation of $J(Q',W')$). For each $(i',j')\in P'$, let $\gamma_{i'j'}^\ast= \beta_{j'}\alpha_{i'}: Y_{i'}\to Y_{j'}$. If $\delta:s\to t$ is an arrow in $Q'$ which is not equal to $\alpha_i^\ast,\beta_j^\ast$ or $\gamma_{ij}^\ast$ then $\delta$ is also an arrow in $Q$ and we define $\delta:Y_s\to Y_t$ to be the equal to the morphism $\delta:X_s\to X_t$. It remains to define $\alpha_i^\ast:Y_k\to Y_i$ for all $i\in I$.

Given $i\in I$, consider the morphism ${(\gamma_{ij},F_{ij'})}:\bigoplus X_j\to X_i$ given by $\gamma_{ij}$ on $Y_j$ for all $j$ so that $(i,j)\in P$ and by $F_{ij'}$ on $Y_{j'}$ for all $j'$ so that $(i,j')\in P'$. Then, Equation \eqref{partial alpha-i W=0} implies that the composition ${(\gamma_{ij},F_{ij'})}\circ (\beta_j)=0$. Therefore, there is a unique induced map $\alpha_i^\ast:Y_k\to Y_i=X_i$ satisfying the following. (See Equation \eqref{definition of alpha i ast}.)
\begin{enumerate}
\item $\alpha_i^\ast \beta_j^\ast=\gamma_{ij}$ when $(i,j)\in P$,
\item $\alpha_i^\ast \beta_{j'}^\ast=F_{ij'}$ when $(i,j')\in P'$.
\end{enumerate}
\begin{equation}\label{definition of alpha i ast}
\xymatrix{
0 \ar[r] &
	X_k\ar@/_1pc/[rd]_(.35)0\ar[r]^{(\beta_j)} &
	\bigoplus X_j\ar[d]_(.4){(\gamma_{ij},F_{ij'})}\ar[r]^{(\beta_j^\ast)}&
Y_k \ar[r]\ar@{-->}[dl]^(.4){\alpha_i^\ast}& 0\\
 &  & X_i=Y_i &  
	}
\end{equation}


We need to verify that $Y$, with these maps, satisfies the relations for $J(Q',W')$. Note that since $\alpha_i^\ast \beta_j^\ast=\gamma_{ij}$ for any $(i,j)\in P$ and $\gamma_{i'j'}^\ast= \beta_{j'}\alpha_{i'}$ for $(i',j')\in P'$, the map $\lambda:X_s\to X_t$ for any path $\lambda$ in $Q$ with $s,t\neq k$ which does not contain a subpath $\alpha_i \beta_j$ for any $(i,j)\in P$ is equal to the map $\widetilde\lambda^\ast:Y_s\to Y_t$ for $Y$:
\[
	\widetilde\lambda^\ast=\lambda:X_s=Y_s\to X_t=Y_t.
\]
In particular, this implies the required condition that $\widetilde W_\delta^\ast=0$ for $Y$ for any arrow $\delta$ not equal to $\alpha_i^\ast,\beta_j^\ast$ or $\gamma_{i'j'}^\ast$ since $\widetilde W_\delta^\ast=W_\delta=0$ on $X$ and thus on $Y$.

Condition (2) above is equivalent the required condition $\alpha_{i'}^\ast \beta_{j'}^\ast=\widetilde F^\ast_{i'j'}$ for $(i',j')\in P'$ ($\partial_{\gamma_{i'j'}^\ast}W'=0$) since $\widetilde F^\ast_{i'j'}= F_{i'j'}$ on $Y$.

To verify the relation $\partial_{\alpha_i^\ast}W'=0$ we substitute $\gamma_{ij'}^\ast=\beta_{j'}\alpha_i$ for $(i,j')\in P'$ and $\widetilde G_{ij}^\ast=G_{ij}=-\beta_j\alpha_i$ for $(i,j)\in P$ (from \eqref{partial alpha-i W=0}). Then, the required condition $\partial_{\alpha_i^\ast}W'=0$ becomes:
\[
	\sum_{(i,j')\in P'}\beta_{j'}^\ast\beta_{j'}\alpha_i+\sum_{(i,j)\in P}\beta_j^\ast\beta_j\alpha_i=0
\]
But $P\coprod P'=I\times J$. So, this is the same as $\sum_{ij} \beta_j^\ast\beta_j\alpha_i=0$ which follows from the fact that $\sum_j\beta_j^\ast\beta_j=0$.

Similarly, the relation $\partial_{\beta_j^\ast}W'=0$ is equivalent to the condition that $\sum_{ij} \beta_j\alpha_i\alpha_i^\ast=0$. Since the $\beta_j$ together form a monomorphism and the $\beta_j^\ast$ together form an epimorphism, this condition is equivalent to the condition $\sum_i \alpha_i\alpha_i^\ast\beta_j^\ast=0$ for each $j\in J$. But this is equivalent to the condition $\partial_{\beta_j}W=0$ in \eqref{partial alpha-i W=0} using the substitutions (1) and (2) above.

This shows that $Y$ is a representation of $J(Q',W')$ and $Y\in \,^{\perp}S_k'$. By naturality of the cokernel $Y_k$ in \eqref{functorial exact sequence}, the assignment $Y=\psi_k(X)$ defines a functor $\psi_k:S_k^\perp\to \,^{\perp}S_k'$. In the opposite direction, for each $Y\in\,^{\perp}S_k'$, we let $X_k$ be the kernel of the epimorphism $(\beta_j^\ast):\bigoplus Y_j\to Y_k$. Let $\alpha_i:X_i=Y_i\to X_k$ be the unique morphism so that $\beta_{j'}\alpha_i=\gamma_{ij'}^\ast$ for $(i,j')\in P'$ and $\beta_j\alpha_i=-G_{ij}$ for $(i,j)\in P$. Finally, $\gamma_{ij}=\alpha_i^\ast\beta_j^\ast$ for all $(i,j)\in P$. The verification that $X\in S_k^\perp$ is analogous to the above discussion and it is clear that $Y\mapsto X$ gives the inverse of the functor $\psi_k$.
\end{proof}

\begin{lem}\label{lem: rotation lemma for modules}
Using the same notation as in Lemma \ref{lem: inductive step A=B for Q finite type}, let $M_1,\cdots,M_m\in S_k^\perp\subset \Lambda\text-mod$. Then \emph{(1)} $\psi_k(M_1),\cdots,\psi_k(M_m)$ is a FHO sequence in $\Lambda'\text-mod$ which lies in $\,^{\perp}S_k'$ if and only if \emph{(2)} $S_k,M_1,\cdots,M_m$ is a FHO sequence in $\Lambda\text-mod$.
\end{lem}

\begin{proof} Let $\cB=\cF(S_k\oplus M)$ where $M=M_1\oplus\cdots\oplus M_m$. Thus 
\[
	\cB=M^\perp\cap S_k^\perp= \{Y\in S_k^\perp\,|\, \Hom_{\Lambda}(M,Y)=0\}.
\]
We will show that both statements in the lemma are equivalent to the third statement:

(3) $M_1,\cdots,M_m$ is a maximal weakly FHO sequence in $S_k^\perp\cap\,^\perp \cB$.

The equivalence $(1)\Leftrightarrow(3)$ is clear. Since $\psi_k:S_k^\perp\cong \,^{\perp}S_k'$ by Lemma \ref{lem: inductive step A=B for Q finite type}, we have
\[
	\cB':=\psi_k(\cB)=\psi_k(M)^\perp\cap \,^\perp S_k'=\{Y'\in\,^{\perp}S_k'\,|\, \Hom_{\Lambda'}(\psi_k(M),Y')=0\}.
\]
So, $\psi_k(S_k^\perp\cap\,^\perp \cB)=\,^{\perp}S_k'\cap\,^\perp\cB'=\cG(\psi_k(M))$ by Proposition \ref{prop: HN stratification with 3 strata} with $S_k',\psi_k(M),\cB'$ playing the roles of $X,Y,\cC$. Since $\psi_k$ is an isomorphism, $M_1,\cdots,M_m$ is maximal weakly FHO in $S_k^\perp\cap\,^\perp\cB$ if and only if $\psi_k(M_1),\cdots,\psi_k(M_m)$ is maximal weakly FHO in $\psi_k(S_k^\perp\cap\,^\perp \cB)=\cG(\psi_k(M))$. By Definition \ref{def of FHO} this is equivalent to (1).


For the equivalence $(2)\Leftrightarrow(3)$, note that (3) is equivalent to the statement that $S_k$, $M_1,\cdots,M_m$ is weakly FHO in $\Lambda\text-mod$ and that no objects of $\,^\perp\cB=\cG(S_k\oplus M)$ can be inserted between the $M_i$, after $M_m$ or before $M_1$ (and after $S_k$). Since these hold under condition (2) we have $(2)\then (3)$. To show $(3)\then (2)$ it remains to show one more condition: that no object of $\,^\perp\cB$ can be inserted before $S_k$ in the sequence.

Suppose not. Then there is a Schurian $\Lambda$-module $X\in\,^\perp\cB$ so that $\Hom_\Lambda(X,S_k\oplus M)=0$. Let $T\subsetneq X$ be the largest submodule having only $S_k$ in its composition series. Then $\Hom_\Lambda(S_k,X/T)=0$ and $\Hom_\Lambda(X/T,M)=0$. So, $S_k^\perp\cap\,^\perp\cB$ contains $X/T\neq0$. As in the proof of Proposition \ref{prop: HN stratification with 3 strata}, any object $Z$ in $S_k^\perp\cap\,^\perp\cB$ of minimal length is Schurian. Then $Z,M_1,\cdots,M_m$ is a weakly FHO sequence in $S_k^\perp\cap\,^\perp \cB$ contradicting the maximality of the sequence $M_1,\cdots,M_m$. Therefore, all three statements are equivalent.
\end{proof}

\begin{thm}\label{thm: inductive thm implying main thm}
For any $m\ge1$ and any nondegenerate quiver with potential $(Q,W)$ of finite type, there is a $1$-$1$ correspondence:
\[
\left\{\begin{array}{c}
\text{green sequences for $Q$}\\
\text{ of length $m$}
\end{array}
\right\}\cong 
\left\{
\begin{array}{c}
\text{isomorphism classes of}\\
\text{ FHO sequences $M_1,\cdots,M_m$}\\
\text{in $J(Q,W)\text-mod$}
\end{array}
\right\}
\]
where the green sequence corresponding to $(M_i)$ is the unique one with $c$-vectors $\beta_i=\undim M_i$.
\end{thm}

We call the sequence of $c$-vectors of a green mutation sequence a \emph{green $c$-sequence}.

\begin{proof} Since the initial $c$-matrix is the identity matrix $I_n$, the first mutation is mutation at a unit vector which is arbitrary. By Lemma \ref{lem:FHO of length one is simple}, a module occurs as the first module $M_1$ of a FHO sequence if and only if $M_1$ is simple. This proves the theorem in the case $m=1$.

For $m\ge2$, $M_1=S_k$ is simple and we claim that the following are equivalent where $\Lambda'=\mu_k\Lambda=J(Q',W')$ as discussed above and $S_k'$ is the simple $\Lambda'$-module at vertex $k$.
\begin{enumerate}
\item $M_1,\cdots,M_m$ is a FHO sequence in $\Lambda\text-mod$.
\item $\psi_k(M_2),\cdots,\psi_k(M_m)$ is FHO in $\Lambda'\text-mod$ and lies in $\,^\perp S_k'$.
\item $\psi_k(M_2),\cdots,\psi_k(M_m)$ is FHO in $\Lambda'\text-mod$ and $\psi_k(M_i)\not\cong S_k'$ for all $i\ge2$.
\item $\beta_2',\cdots,\beta_m'$, where $\beta_i'=\undim \psi_k(M_i)$, is a green $c$-sequence for $Q'=\mu_kQ$ and $\beta_i'\neq e_k$, the $k$th unit vector in $\ZZ^n$, for all $i\ge2$.
\item $e_k,\beta_2,\beta_3,\cdots,\beta_m$ is a green $c$-sequence for $Q$ where $\beta_i=\varphi_k(\beta_i')$ for all $i\ge2$.
\end{enumerate}
$(1)\ifff(2)$ the same statement as Lemma \ref{lem: rotation lemma for modules} since $M_1=S_k$.

\noindent$(2)\ifff(3)$ by Lemma \ref{lem:when Sk is not one of the Mi}.

\noindent$(3)\ifff(4)$ is the theorem for $m-1$ applied to $\Lambda'\text-mod$ with the additional condition that $\psi_k(M_i)\not\cong S_k'$ which is equivalent to $\beta_i'\neq e_k$ since $\beta_i'=\undim\psi_k(M_i)$.

\noindent$(4)\ifff(5)$ by the well-known mutation formula for $c$-vectors. See, e.g., \cite[Thm 2.1.8]{BHIT}.

Finally, $\beta_i=\varphi_k(\undim \psi_k(M_i))=\undim M_i$ by Lemma \ref{lem: inductive step A=B for Q finite type} since $\varphi_k$ is an involution.
\end{proof}

\begin{cor}\label{cor: main thm of section 2}
For any nondegenerate quiver with potential $(Q,W)$ of finite type, there is a $1$-$1$ correspondence between maximal green sequences for $Q$ and isomorphism classes of complete FHO sequences $M_1,\cdots,M_m$ in $J(Q,W)\text-mod$ where the green $c$-sequence corresponding to $(M_i)$ is $(\beta_i=\undim M_i)$. \qed
\end{cor}

\begin{cor}[Rotation Lemma]\label{cor: rotation lemma}
There is a $1$-$1$ correspondence between complete FHO sequences $M_1,\cdots,M_m$ in $J(Q,W)$ starting with $M_1=S_k$ and complete FHO sequences $X_1,\cdots,X_m$ of the same length for $J(\mu_k(Q,W))$ ending with $X_m=S_k'$. The correspondence is given by $X_i=\psi_k(M_{i+1})$ for $i<m$.
\end{cor}

\begin{proof} The correspondence is already given by $X_i=\psi_k(M_{i+1})$ for $i<m$ and $M_j=\psi_k^{-1}(X_{j-1})$ for $j>1$. It remains to show that this formula sends complete FHO sequences in $\Lambda\text-mod$ to complete FHO sequences in $\Lambda'\text-mod$ and vice versa. But this follows from Proposition \ref{prop: complete iff M1 is simple and M2, etc is maximal} and the fact that $\psi_k:S_k^\perp\cong \,^\perp S_k'$ is an equivalence of categories: $M_1,\cdots,M_m$, with $M_1=S_k$ and $M_2,\cdots,M_m\in M_1^\perp$ is a complete FHO sequence for $\Lambda\text-mod$ if and only if $M_2,\cdots,M_m$ is a maximal weakly FHO sequence in $S_k^\perp$. Since $\psi_k$ is an equivalence, this is equivalent to $X_1,\cdots,X_{m-1}$ being a maximal weakly FHO sequence in $\,^\perp S_k'$ in $\Lambda'\text-mod$. By \ref{prop: complete iff M1 is simple and M2, etc is maximal} this is equivalent to $X_1,\cdots,X_{m-1},S_k'$ being a complete FHO sequence in $\Lambda'\text-mod$.
\end{proof}

\subsection{Interated mutation of forward hom-orthogonal sequences}

For the next section we need the following iterated version of Lemmas \ref{lem: inductive step A=B for Q finite type} and \ref{lem: rotation lemma for modules}.

\begin{prop}\label{prop: iterated mutation}
Let $M_1,\cdots,M_m$ be a FHO sequence in $\Lambda\text-mod$ where $\Lambda=J(Q,W)$ is of finite representation type. Let $(k_1,\cdots,k_m)$ be the corresponding mutation sequence of $(Q,W)$. Let
\[
	\Lambda'=\mu_{k_m}\cdots\mu_{k_1}\Lambda=J(\mu_{k_m}\cdots\mu_{k_1}(Q,W)).
\]
Then $\exists N\in\Lambda'\text-mod$ and an equivalence of full subcategories $\psi:M^\perp\cong \,^{\perp}N$ where $M=M_1\oplus\cdots\oplus M_m$ and a linear automorphism $\varphi$ of $\ZZ^n$ so that
\[
	\undim \psi(X)=\varphi(\undim X).
\]
Furthermore, given Schurian $\Lambda$-modules $X_1,\cdots,X_s\in M^\perp$, the sequence $\psi(X_1),\cdots,\psi(X_s)$ is FHO in $\Lambda'$-$mod$ if and only if $M_1,\cdots,M_m,X_1,\cdots,X_s$ is FHO in $\Lambda\text-mod$.
\end{prop}

\begin{proof}
Lemmas \ref{lem: inductive step A=B for Q finite type} and \ref{lem: rotation lemma for modules} give the case $m=1$. Suppose by induction that the proposition holds for $m$ with $m\ge1$. Let $M_{m+1}$ be one choice for the next term of the FHO sequence $M_1,\cdots,M_m$. By induction $\psi(M_{m+1})$ is a singleton FHO sequence and thus $\psi(M_{m+1})=S_k'$ is a simple $\Lambda'$-module by Lemma \ref{lem:FHO of length one is simple}. So, the equivalence $\psi:M^\perp \cong \,^{\perp}N$ given by induction on $m$ sends $M_{m+1}$ to $S_k'\in \,^\perp N$ and we have
\[
	\psi:M^\perp \cap M_{m+1}^\perp\cong \,^{\perp}N\cap S_k'^{\perp}.
\]

Let $\Lambda''=\mu_k\Lambda'$. Then, by Lemma \ref{lem: inductive step A=B for Q finite type}, we have an equivalence $\psi_k:S_k'^{\perp}\cong \,^{\perp}S_k''\subset \Lambda''$-$mod$ and $\undim \psi_k(X)=\varphi_k(\undim X)$ for all $X\in S_k'^{\perp}$. But, $S_k'\in\,^{\perp}N$ implies $N\in S_k'^{\perp}$. And $\psi_k$ restricts to an equivalence
\[
	\psi_k: \,^{\perp}N\cap S_k'^{\perp}\cong \,^{\perp}\psi_k(N)\cap \,^{\perp}S_k''.
\]
Combining these gives the required equivalence
\[
	\psi_k\psi:(M_1\oplus\cdots \oplus M_{m+1})^\perp\cong \,^{\perp}(\psi_k(N)\oplus S_k'').
\]
On dimension vectors this is
\[
	\undim \psi_k\psi (X)=\varphi_k(\undim \psi (X))=\varphi_k\varphi(\undim X)
\]
where $\varphi_k\varphi$ is a composition of two automorphisms of $\ZZ^n$.

Finally, suppose that $X_1,\cdots,X_s$ are Schurian modules in $(M\oplus M_{m+1})^\perp$. Then we are required to show that the following are equivalent.
\begin{enumerate}
\item $\psi_k\psi(X_1),\cdots,\psi_k\psi(X_s)$ is a FHO sequence in $\Lambda''$-$mod$.
\item $M_1,\cdots,M_{m+1},X_1,\cdots,X_s$ is FHO in $\Lambda$-$mod$.
\end{enumerate}
But, by Lemma \ref{lem: rotation lemma for modules}, (1) is equivalent to $\psi(M_{m+1}),\psi(X_1),\cdots,\psi(X_s)$ being a FHO sequence in $\Lambda'$-$mod$. By induction on $m$, this is equivalent to (2). So, all statements hold for $m+1$ and we are done.
\end{proof}


\section{Semistability sets for algebras of finite representation type}

In this section we prove that complete FHO sequences for $\Lambda=J(Q,W)$ of finite representation type are given by wall crossing sequences for (generic) ``green paths''. One direction is known, namely that a green path $\gamma$ gives a complete FHO sequence assuming that each wall crossed by $\gamma$ supports a unique Schurian module. (See \cite{PartI}, Theorems 3.6, 3.8.) Conversely, for any maximal green sequence for $Q$ we construct a green path. Br\"ustle, Smith and Treffinger \cite{BST} have recently shown the analogous statement for any finite dimensional algebra using $\tau$-tilting, namely that maximal green sequences defined using $\tau$-tilting are equivalent to wall crossing sequences which are also equivalent to finite Harder-Narasimhan stratifications of the module category. Demonet, Iyama and Jasso \cite{DIJ} have obtained similar results for $\tau$-tilting finite algebras.

\subsection{Basic definitions} We present here the basic definitions and proofs of well-known statements in detail for the benefit of our students.

Suppose that $\Lambda$ is a finite dimensional algebra over a field $K$. For every nonzero module $M$, the \emph{semistability set} $D(M)\subset \RR^n$ of $M$ is given by:
\[
	D(M):=\{x\in\RR^n\,|\, x\cdot \undim M=0, x\cdot \undim M'\le 0\ \forall M'\subset M\}
\]
This is clearly a closed convex subset of the hyperplane $H(M)=\undim M^\perp$ of all $x\in\RR^n$ perpendicular to $\undim M$. If $x\in D(M)$ and $x\cdot \undim M'=0$ for some $M'\subsetneq M$ then clearly $x\in D(M')$. Let $\partial D(M)$ be the set of all $x\in D(M)$ so that $x\in D(M')$ for some $M'\subsetneq M$. Let $int\,D(M)=D(M)-\partial D(M)$. This is a (possibly empty) open subset of the hyperplane $H(M)$. We call $int\,D(M)$ the \emph{stable set} of $M$. Clearly, we have:
\[
	int\,D(M)=\{x\in\RR^n\,|\, x\cdot \undim M=0, x\cdot \undim M'< 0\ \forall M'\subsetneq M\}.
\]

\begin{prop}\label{prop: int D(M) is empty for nonSchurian M}
If $M$ is not Schurian then $D(M)$ is contained in $D(M')$ for some proper submodule $M'\subsetneq M$ and $int\,D(M)$ is empty.
\end{prop}

\begin{proof}
If $M$ is not Schurian there is a nonzero endomorphism $f$ of $M$ which is not an isomorphism. Let $K,L$ be the kernel and image of $f$. Then, $\undim K+\undim L=\undim M$. For every $x\in D(M)$ the conditions $x\cdot \undim K,x\cdot \undim L\le0$ and $x\cdot\undim M=0$ imply that $x\cdot \undim K=x\cdot \undim L=0$. Therefore, $D(M)\subset D(K)\cap D(L)$.
\end{proof}

\begin{rem}\label{added remark}
The converse of Proposition \ref{prop: int D(M) is empty for nonSchurian M} is not true in general. For example, take the quiver with
\[
\xymatrixrowsep{10pt}\xymatrixcolsep{10pt}
\xymatrix{
Q: &1\ar[rr]\ar@/^.5pc/[rr] &&2\ar[dl]\\
&& 3\ar[lu]
	}
\]
with relations $rad^5=0$. The two paths $1\to 2\to 3$ give two hom-orthogonal modules $A,B$ with the same dimension vector $(1,1,1)$ and the arrow $3\to 1$ gives an extension $A\into M\onto B$ which is Schurian with $\undim M=(2,2,2)$ and $int\,D(M)=\emptyset$ since $D(M)\subset D(A)$.
\end{rem}

\begin{eg}
Let $\Lambda=KQ/I$ be given by the cyclic quiver 
\[
\xymatrixrowsep{10pt}\xymatrixcolsep{10pt}
\xymatrix{
1\ar[r] &2\ar[d]\\
4\ar[u] & 3\ar[l]
	}
\]
with relations $rad^5=0$. Let $M=P_1=I_1$ with dimension vector $\undim M=(2,1,1,1)$. The simple module $S_1$ is both a submodule and quotient module of $M$ with complementary sub/quotient module $X=M/S_1$ with dimension vector $\undim X=(1,1,1,1)$. Therefore, $D(M)$ is contained in the codimension 2 subspace $H(S_1)\cap H(X)$ of $\RR^4$.
\end{eg}

We say that $D(M)$ has \emph{full rank} if it contains $n-1$ elements which are linearly independent over $\RR$, i.e., it does not lie in the intersection of two distinct hyperplanes as in the example above.

We consider the union $L(\Lambda)=\bigcup D(M)$ of all $D(M)$. For any $x\in L(\Lambda)$, let $M$ be a module of minimal length so that $x\in D(M)$. Then $x\in int\,D(M)$. It follows that $L(\Lambda)$ is a union of the stable sets $int\,D(M)$ for $M$ Schurian. Since there are only countably many hyperplanes of the form $H(M)$, the set $L(\Lambda)$ has measure zero and its complement is dense in $\RR^n$. One example is: any point $x\in \RR^n$ with coordinates linearly independent over $\QQ$ cannot lie on any hyperplane $H(M)$ since each such hyperplane is defined by a linear equation with integer coefficients. We call such points \emph{generic}. It is easy to see that, given any generic point $x$, the path 
\[
\gamma_x(t)=x+(t,t,\cdots,t)
\]
does not pass through the intersection of two distinct hyperplanes $H(M)\cap H(N)$. 

Similarly, a \emph{generic point of $D(M)$} will mean a point having $n-1$ coordinates linearly independent over $\QQ$. The set $D(M)$ contains generic elements if and only if it has full rank. If $x\in D(M)$ is generic then the path $\gamma_x$ will contain a generic point in $\RR^n$ (take $\gamma_x(t)$ where $t$ is $\QQ$-linearly independent from the coordinates of $x$). Thus, $\gamma_x$ cannot meet the intersection of two distinct hyperplanes $H(M)\cap H(N)$ and the intersection of $\gamma_x$ with any other $D(N)$ will also be generic. When $D(M)$ has full rank, its generic elements form a dense subset. (The nongeneric points in $D(M)$ lie in a countable union of codimension one subsets.)

For any $x_0\in \RR^n$ let $\cW(x_0)$ be the full subcategory of $\Lambda$-$mod$ of all modules $X$ so that $x_0\in D(X)$. It is well-known that $\cW(x_0)$ is an abelian category. (See \cite[Lemma 5.2]{Bridgeland} where the role of $\cW(x_0)$ is played by $\cP(\phi)$.)

\begin{prop}\label{prop: generic x means W(x) has only one vector}
If $x_0$ is a generic point of $D(M)$ then every object in $\cW(x_0)$ has dimension vector a rational multiple of $\undim M$.
\end{prop}

\begin{proof}
Since $n-1$ of the coordinates of $x_0$ are linearly independent over $\QQ$, any two rational vectors perpendicular to $x_0$ are proportional to each other.
\end{proof}

\subsection{Green paths} We recall the definition of a green path using generic points.

\begin{defn}\label{def: green path}
By a \emph{generic path} for $\Lambda$ we mean a smooth path $\gamma:\RR\to\RR^n$ which meets each set $D(M)$ at a finite number of points all of which are generic. The path will be called \emph{green} if all coordinates of $\gamma(t)$ are positive, resp. negative, for $t>>0$, resp. $t<<0$ and, whenever $\gamma(t_0)\in D(M)$ the velocity vector of $\gamma$ points in the positive direction, i.e.,
\[
	\frac {d\gamma}{dt}(t_0)\cdot \undim M>0.
\]
For example, the linear path $\gamma_x$ is a generic green path for any generic $x\in\RR^n$.
\end{defn}

\begin{defn}
A module $M$ is \emph{strongly Schurian} if $M$ is Schurian and if, for any generic point $x_0\in D(M)$, any module $X$ so that $x_0\in D(X)$ is an iterated self-extension of $M$. In particular, $M$ is the only Schurian module in $\cW(x_0)$.
\end{defn}

For example, any simple module is strongly Schurian. The following theorem, proved in Theorem 5.13 in \cite{PartI} is due to Bridgeland \cite{Bridgeland} in a different language.

\begin{thm}\label{thm: HN filtration} Let $\Lambda$ be any finite dimensional algebra over $K$ and let $\gamma$ be a generic green path for $\Lambda$. Then, for any $\Lambda$-module $X$ there is a unique filtration $0=X_0\subset X_1\subset \cdots\subset X_m=X$ so that each $X_i/X_{i-1}\in\cW(\gamma(t_i))$ for some $t_1<\cdots<t_m$ so that $\gamma(t_i)t_i\in D(M_i)$ for Schurian modules $M_i$. In particular, $\Hom_\Lambda(M_i,M_j)=0$ for $i<j$.\qed
\end{thm}

$0=X_0\subset X_1\subset \cdots\subset X_m=X$ is called the \emph{Harder-Narasimhan (HN) filtration} of $X$. The last sentence in Theorem \ref{thm: HN filtration}, usually stated without proof, follows from the uniqueness of the HN-filtration for $X=M_i\oplus M_j$. Any morphism $f:M_i\to M_j$ gives a different filtration using the graph of $f$ as the submodule $X_i$. So, $f$ must be unique, i.e., $f=0$.

\begin{cor}\label{cor: first wall is simple}
Let $D(M_1)$, with $M_1$ Schurian, be the first wall crossed by $\gamma$, a generic green path for $\Lambda$, i.e., $\gamma(t_1)\in D(M_1)$ and $\gamma(t)\notin L(\Lambda)$ for any $t<t_1$. Then $M_1$ is simple.
\end{cor}

\begin{proof}
Apply Theorem \ref{thm: HN filtration} to any simple quotient module $X$ of $M_1$. Then $X\in \cW(\gamma(t_i))$ for some $t_i\ge t_1$. Then $t_i=t_1$ since, otherwise, $\Hom_\Lambda(M_1,X)=0$, contradicting the hypothesis that $X$ is a quotient of $M_1$. By Proposition \ref{prop: generic x means W(x) has only one vector} each object in $\cW(\gamma(t_1))$ has dimension vector a multiple of the unit vector $\undim X$. Thus it must be an iterated self-extension of $X$. So, $X$ is the only Schurian object in $\cW(\gamma(t_1))$. So, $M_1=X$.
\end{proof}

The following is Theorem 3.8 in \cite{PartI}.

\begin{prop}\label{prop: wall sequence is hom orthog}
Suppose that $\gamma$ is a generic green path for $\Lambda$ which meets only a finite number of walls $D(M_1),\cdots,D(M_m)$ in that order. So, there exist $t_1<t_2<\cdots<t_m$ so that $\gamma(t_i)\in D(M_i)$ and $\gamma(t)\notin L(\Lambda)$ for any other $t\in\RR$. Suppose that $M_i$ is the unique Schurian module in $\cW(\gamma(t_i))$ for $i\le k$ (so all other objects of $\cW(\gamma(t_i))$ are iterated self-extensions of $M_i$). Then 
\begin{enumerate}
\item The torsion class $\cG(M_1\oplus\cdots\oplus M_k)$ consists of all modules $X$ whose HN-filtration has nonzero subquotients only in $\cW(\gamma(t_1)),\cdots,\cW(\gamma(t_k))$.
\item $M_1,\cdots, M_k$ is a FHO sequence for $\Lambda$. 
\item
If $k=m$ then $M_1,\cdots,M_m$ is a complete FHO sequence for $\Lambda$.
\end{enumerate}
\end{prop}

\begin{proof} (1) If $X$ has such a filtration, then each subquotient of $X$ will be an iterated self extensions of $M_i$ for some $i\le k$. So, $X\in\cG(M)$. Conversely, if the HN-filtration of $X$ goes beyond $\cW(\gamma(t_k))$, the last term will be a nonzero quotient of $X$ in some $\cW(\gamma(t_p))$ for $p>k$. But, $\cW(\gamma(t_p))\subset M^\perp=\cF(M)$. So, $X\notin \cG(M)$.

(2) We know that $M_1,\cdots,M_k$ is weakly FHO and lies in $\cG(M)$. To show that it is maximal in $\cG(M)$ suppose that $X$ is any Schurian object in $\cG(M)$ not isomorphic to any $M_i$. By assumption $X$ does not lie in $\cW(\gamma(t_i))$ for $i\le k$. So, the HN-filtration of $X$ gives a submodule $X_i\in\cW(\gamma(t_i))$ and quotient module $X/X_{j-1}\in \cW(\gamma(t_j))$ for some $i<j\le k$. So, $X$ cannot be inserted into the sequence $M_1,\cdots,M_k$, since it must go before $M_i$ and after $M_j$. So, $M_1,\cdots,M_k$ is maximal in $\cG(M)$ and thus a FHO sequence.


(3) This follows from (2) and the definition of ``complete FHO sequence''.
\end{proof}

We will show that this observation applies to every generic green path for $\Lambda=J(Q,W)$ of finite representation type.

\begin{thm}\label{thm: walls are strongly Schurian}
Let $\Lambda=J(Q,W)$ be a Jacobian algebra of finite representation type. Let $\gamma$ be any generic green path for $\Lambda$. Let $D(M_1),\cdots,D(M_m)$ be the walls crossed by $\gamma$ where each $M_i$ is Schurian. Then each $M_i$ is strongly Schurian. In particular, $M_1,\cdots,M_m$ is a complete FHO sequence for $\Lambda$.
\end{thm}

\begin{proof} For any generic point $x\in D(M)$, the green path $\gamma_x$ passes through $x$. Thus, it suffices to show that, for any generic green path $\gamma$ passing through $D(M_k)$ at $t=t_k$, $\cW(\gamma(t_k))=add\,M_k$. This holds for $k=1$ since $M_1$ is simple. Let $\cW_j=\cW(\gamma(t_j))$ and suppose by induction on $k$ that $\cW_j=add\,M_j$ for $j\le k$. By Proposition \ref{prop: wall sequence is hom orthog}, it follows that $M_1,\cdots,M_{k}$ is a FHO sequence in $\Lambda$-$mod$. We will use Proposition \ref{prop: iterated mutation} with $s=1$ to show that $\cW_{k+1}=add\,M_{k+1}$.

Let $M=M_1\oplus\cdots \oplus M_k$. Let $\cF$ be the extension closed full subcategory of $\Lambda\text-mod$ of all $X$ whose HN-filtration has subquotients only in $\cW_j$ for $j\ge k+2$. Then $\cF\subset M^\perp$ and $\cW_{k+1}=M^\perp\cap \,^\perp\cF$ by Theorem \ref{thm: HN filtration}.

Using the notation of Proposition \ref{prop: iterated mutation}, let $\Lambda'$ be the iterated mutation of $\Lambda$ corresponding to $M$ and take $N\in \Lambda'\text-mod$ so that there is an equivalence of categories 
\[
\psi:M^\perp\cong \,^{\perp}N.\]
Then we claim that $Y=\psi(M_{k+1})\in\,^{\perp}N$ is a simple $\Lambda'$-module and $\psi(\cW_{k+1})=add\,Y$.

To see this, let $S$ be any simple quotient of $Y$. Since $\,^{\perp}N$ is closed under quotients, $S\in\,^{\perp}N$. But $Y\in \,^{\perp}\psi(\cF)$ implies $S\in \,^{\perp}\psi(\cF)$ which implies that $\psi^{-1}(S)\in M^\perp\cap \,^\perp \cF=\cW_{k+1}$. By Proposition \ref{prop: generic x means W(x) has only one vector} the dimension vectors of $\psi^{-1}(S)$, $M_{k+1}$ and any $W\in \cW_{k+1}$ are proportional. By Proposition \ref{prop: iterated mutation}, the dimension vectors of $S$, $Y$ and $\psi(W)$ are proportional. Since $S$ is simple, these all lie in $add\,S$. So $\psi(\cW_{k+1})=add\,S$ and $S$ is the unique Schurian object of $\psi(\cW_{k+1})$. Since $M_{k+1}\in \cW_{k+1}$ is Schurian, $\psi(M_{k+1})\cong S$ is simple as claimed. Since $\psi$ is an equivalence of categories this implies that $\cW_{k+1}=add(M_{k+1})$ proving the theorem.
\end{proof}

\subsection{Compartments of $L(\Lambda)$}\label{sec33}

Since $\Lambda=J(Q,W)$ has finite representation type, $L(\Lambda)=\bigcup D(M)$ is closed. So, its complement $\RR^n-L(\Lambda)$ is a disjoint union of connected open sets $\cU$ which we call \emph{compartments}. We will associate a torsion pair $(\cG(\cU),\cF(\cU))$ to each compartment and use these to prove the converse of Theorem \ref{thm: walls are strongly Schurian}. Namely, any complete FHO sequence for $\Lambda$ is given by a generic green path.

\begin{thm}
Let $\Lambda$ be of finite representation type. Then each compartment of $L(\Lambda)$ is convex.
\end{thm}

\begin{proof}
Order the Schurian modules $M_i$ according to dimension. Thus $M_1,\cdots,M_n$ are the simple modules. For each $k\ge n$, let 
\[
L_k(\Lambda)=D(M_1)\cup D(M_2)\cup\cdots\cup D(M_k).
\]
Then we show, by induction on $k\ge n$, that each component of $\RR^n-L_k(\Lambda)$ is convex. Since $D(M_1),\cdots,D(M_n)$ are hyperplanes, this statement holds for $k=n$.

Let $\cU$ be any component of $\RR^n-L_k(\Lambda)$. By induction, $\cU$ is convex and open.
\vs2 

\underline{Claim}: \emph{$\cU\cap D(M_{k+1})$ is either empty or equal to $\cU\cap H(M_{k+1})$.} \vs2

In either case, $\cU-L_{k+1}(\Lambda)$ is a disjoint union of convex open sets, proving the theorem. Thus, it suffices to prove the claim. But this is an easy topological argument. If $\cU\cap D(M_{k+1})$ is nonempty then it is a subset of the connected set $\cU\cap H(M_{k+1})$ which is both (relatively) open and closed. It is closed since $D(M_{k+1})$ is a closed set. It is open since $\partial D(M_{k+1})\subset L_k(\Lambda)$. Therefore, $\cU\cap D(M_{k+1})=\cU\cap H(M_{k+1})$ proving Claim and Theorem.
\end{proof}

\begin{lem}
For any $\Lambda$ of finite representation type, each compartment $\cU$ of $L(\Lambda)$ has at least $n$ walls $D(M_i)$.
\end{lem}

\begin{proof}
Suppose not and let $D(M_i)$ be the walls of $\cU$. Let $\cV\subset \cU$ be one component of the complement of the hyperplanes $H(M_i)$. Then $\cap H(M_i)$ contains a nonzero vector $x$ and its negative $-x$. These lie in the closure of $\cV$. Since $x\neq0$ it has at least one nonzero coordinate $x_j\neq0$. By symmetry assume $x_j>0$. So, $cV$ and thus $\cU$ contains two points $y$, $z$ so that $y_j>0$ and $z_j<0$. But this is impossible since these two points are separated by the hyperplane $D(S_j)=H(S_j)$.
\end{proof}

\begin{defn}
Let $\Lambda$ be an algebra of finite representation type and let $\cU$ be any compartment of $L(\Lambda)$. We say that a generic green path $\gamma$ is \emph{centered} at $\cU$ if $\gamma[0,1]\subset \cU$. For any such path call the path $\gamma(t),t\le0$ the \emph{left part} of $\gamma$. The \emph{right part} is $\gamma(t),t\ge1$. 
\end{defn}

We observe that, since $\cU$ is convex, the left part of any generic green path $\gamma$ centered at $\cU$ can be \emph{spliced} together with the right part of any other generic green path $\gamma'$ centered at $\cU$ to give a new generic green path centered at $\cU$ with left part the same as $\gamma$ and right part the same as $\gamma'$.

We get the following well-known bijection between support tilting objects for $\Lambda\text-mod$ and finitely generated torsion classes by way of the compartments of $L(\Lambda)$.

\begin{lem}
For $\Lambda=J(Q,W)$ of finite representation type and $\cU$ any component of $L(\Lambda)$. Let $\gamma$ be any generic green path centered at $\cU$. Let $D(M_1),\cdots,D(M_k)$ be the walls crossed by the left part of $\gamma$. Let $\cF=M^\perp$ and $\cG=\,^\perp \cF$ where $M=M_1\oplus\cdots\oplus M_k$. Let $C$ be the $c$-matrix of the mutation sequence corresponding to the FHO sequence $M_1,\cdots,M_k$. Then $C$ and the torsion pair $(\cG,\cF)$ depend only on $\cU$ and are independent of the choice of $\gamma$ up to permutation of the columns of $C$. 
\end{lem}

\begin{proof}
Let $D(M_{k+1}),\cdots,D(M_m)$ be the walls crossed by the right part of $\gamma$. Then $\cG=\,^\perp M'$ and $\cF=\cG^\perp$ where $M'=M_{k+1}\oplus \cdots\oplus M_m$. Therefore $(\cG,\cF)$, defined using the left part of $\gamma$, depends only on the right part of $\gamma$. Given two paths $\gamma,\gamma'$ centered at $\cU$ splice them to get $\gamma''$. Then $\gamma,\gamma''$ give the same torsion pair since they have the same right part and $\gamma',\gamma''$ give the same torsion pair since they have a common left part. So, $\gamma,\gamma'$ give the same torsion pair. 

Similarly, the matrix $C$, defined using the left part of $\gamma$, is determined, up to permutation of its columns, by the right part of $\gamma$ since it can be obtained by backward mutation from the final exchange matrix which is the same as the initial exchange matrix with $-I_n$ as $c$-matrix.
\end{proof}

The columns of $C$ will be called the \emph{c-vectors} of $\cU$. The torsion class of $\cU$ is denoted $\cG(\cU)$. A wall $D(M)$ of $\cU$ will be called \emph{positive} if $\cU$ is on the negative side of $D(M)$, i.e., $(x-y)\cdot\undim M<0$ for any $x\in \cU$ and $y\in D(M)$. Otherwise, the wall is \emph{negative}.

\begin{thm}\label{thm 3} \emph{(a)} Each positive wall of $\cU$ is $D(M)$ where $\undim M$ is a $c$-vector of $\cU$.

\emph{(b)} Each negative wall of $\cU$ is $D(N)$ where -$\undim N$ is a $c$-vector of $\cU$.
\end{thm}

\begin{proof}
(b) Let $D(N_0)$ be a negative wall of $\cU$. Let $\cV$ be the region opposite $D(N_0)$. Choose a path from $\cV$ to $\cU$ and extend to a green path centered at $\cU$. By Theorem \ref{thm: walls are strongly Schurian} this gives a complete FHO sequence $M_1,\cdots,M_k,N_0,N_1,\cdots,N_m$ so that $\cG(\cU)=\cG(M_1\oplus\cdots\oplus M_k\oplus N_0)$. This corresponds to a green $c$-sequence $\beta_1,\cdots,\beta_k,\alpha_0,\alpha_1,\cdots, \alpha_m$ where $\beta_i=\undim M_i$, $\alpha_j=\undim N_j$. So, $\undim N_0$ is a positive $c$-vector for $\cV$ and a negative $c$-vector of $\cU$. The proof of (a) is the same with $\cV,\cU$ reversed.
\end{proof}

Since $\cU$ has at least $n$ walls and at most $n$ $c$-vectors we get the following.

\begin{cor}
For $\Lambda=J(Q,W)$ of finite representation type, each compartment $\cU$ has exactly $n$ walls $D(M_i)$ where $\undim M_i$ are the $c$-vectors of $\cU$ up to sign.\qed
\end{cor}

\begin{cor}\label{cor: formerly 3.16}
For every FHO sequence $M_1,\cdots,M_k$ for $\Lambda$ there is a generic green path whose left part passed through the walls $D(M_1),\cdots,D(M_k)$.
\end{cor}

\begin{proof}
Let $M_1,\cdots,M_{k}$ is a FHO sequence. Then, by induction on $k$, there is a green path $\gamma$ centered in some compartment $\cU$ whose left part passes through $D(M_1),\cdots,D(M_{k-1})$. Since FHO sequences correspond to mutation sequences, one of the positive $c$-vectors of $\cU$ must be $\undim M_k$. So, $D(M_k)$ is one of the positive walls of $\cU$. Let $\cV$ be the compartment on the other side of this wall. Then the right part of $\gamma$ can be modified so that it first passes though $D(M_k)$. This completes the induction.
\end{proof}

Combining this with Theorem \ref{thm: walls are strongly Schurian} we get the main result of this section:

\begin{thm}\label{main thm of section three}
Let $\Lambda=J(Q,W)$ be a Jacobian algebra of finite representation type and let $M_1,\cdots,M_m$ be a sequence of Schurian $\Lambda$-modules. Then $M_1,\cdots,M_m$ is a complete FHO sequence if and only if there exists a generic green path which passes through the walls $D(M_1),\cdots,D(M_m)$ in that order.
\end{thm}

Together with Corollary \ref{cor: main thm of section 2} this completes the proof of Theorem \ref{main thm} from the introduction.

\section{Appendix: maximal length MGSs}

This Appendix is joint work of the author with Gordana Todorov.


In this section we consider maximal green sequences of maximal length for cluster-tilted algebras of finite representation type. We describe an upper bound and a lower bound for this maximal length. We give three examples and a conjecture, namely, that the lower bound is sharp. The second example, \ref{eg: Dn}, is joint work with PJ Apruzzese. The third example, \ref{eg: Garver and I}, is joint work with Al Garver.

We use the well-known fact that cluster-tilted algebras are \emph{relation-extensions} of tilted algebras \cite{ABS} and we use the result of \cite{A} describing all tilted algebras of type $A_n$. This is Theorem \ref{Assem's theorem} below. See the lecture notes \cite{A2} for more details about this topic.

If $H$ is a hereditary algebra with $n$ simple modules, a \emph{tilting module} is a module $T$ with $n$ nonisomorphic indecomposable summands which is \emph{rigid}, i.e., so that $\Ext_H^1(T,T)=0$. The endomorphism ring of $T$ is called a \emph{tilted algebra}. In the cluster category $\cC_H$ of $H$ introduced in \cite{BMRRT}, a \emph{cluster-tilting object} is a rigid object $T\in \cC_H$ with $n$ nonisomorphic indecomposable summands. The endomorphism ring of $T$ as an object in $\cC_H$ is called a \emph{cluster-tilted algebra}. 

Every tilting module $T$ for $H$ can also be viewed as a cluster-tilting object in $\cC_H$ under a natural inclusion functor $H\text-mod\into \cC_H$. This is a faithful but not full embedding. It sends rigid modules to rigid objects. In this case, it is shown in \cite{ABS} that the cluster-tilted algebra $\Lambda=\End_{\cC_H}(T)$ is a trivial extension of the tilted algebra $C=\End_H(T)$:
\[
	\Lambda\cong C\ltimes \Ext^2_C(DC,C)
\] 
where $DC=\Hom_K(C,K)$ is the dual of $C$. This is called the \emph{relation-extension} of $C$.

When the hereditary algebra $H$ is the path algebra of an acyclic quiver, it was shown in \cite{BIRSm} that each cluster-tilted algebra $\Lambda=\End_{\cC_H}(T)$ is isomorphic to the Jacobian algebra of a quiver with nondegenerate potential. Restricting to quivers of finite type, the converse also holds, i.e., Jacobian algebras of quivers with potential of finite type in the sense of \cite{FZ2} are mutation equivalent to path algebras of Dynkin quivers and are therefore isomorphic to cluster-tilted algebras of finite representation type. This follows from the finite type classification of cluster algebras \cite{FZ2} and translation into the language of quivers with potential using \cite{DWZ}.

\subsection{Tilted and cluster-tilted algebras of type A}\label{ss34:tilted An}

Tilted algebras of type $A$ were classified by Assem \cite{A}. Iterated tilted algebras of type $A$ were classified in \cite{AH}. Cluster-tilted algebras of type $A$ came later \cite{CCS}, \cite{BV}. However, it is easier to start with the cluster-tilted algebras.

As observed in \cite[2.5]{A2}, a cluster-tilted algebra of type $A$ is given by a finite connected full subquiver of the following infinite quiver (an infinite array of oriented 3-cycles attached together at their vertices) modulo the relation that the composition of any two arrows in any oriented 3-cycle is zero.
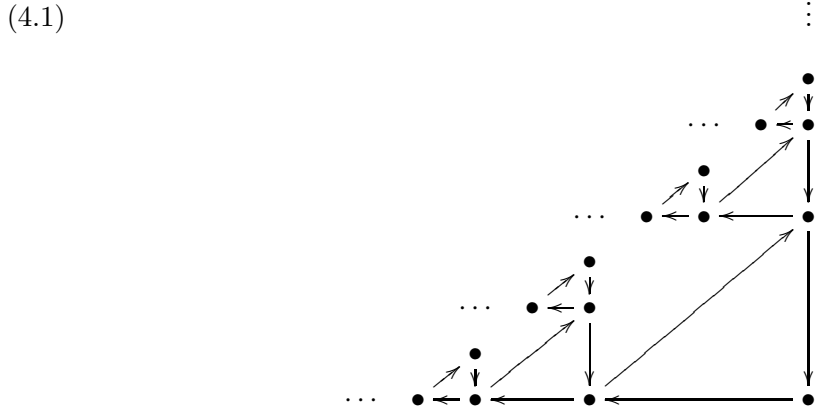
\begin{equation}\label{infinite quiver}
\xymatrixrowsep{6pt}\xymatrixcolsep{6pt}
\xymatrix{
&&&&&&&&\vdots\\
&&&&&&&&\bullet\ar[d]\\
&&&&&& \cdots &\bullet\ar[ur]&\bullet\ar[l]\ar[dd]\\
&&&&&& \bullet\ar[d] \\
&&&& \cdots&\bullet\ar[ur] &\bullet\ar[l]\ar[uurr]&&\bullet\ar[ll]\ar[dddd]\\ 
&&&&\bullet\ar[d]\\
&& \cdots &\bullet\ar[ur]&\bullet\ar[l]\ar[dd]\\
& & \bullet\ar[d] \\
\cdots &\bullet\ar[ur] &\bullet\ar[l]\ar[uurr]&&\bullet\ar[ll]\ar[uuuurrrr]&&&&\bullet\ar[llll] \\ 
	}
\end{equation}
Observe that a full subquiver $Q$ containing two arrows in an oriented 3-cycle will contain the entire 3-cycle. Thus, any cluster-tilted algebra of type $A_n$ is given by $\Lambda=J(Q,W)$ where $Q$ is a connected full $n$ vertex subquiver of the above infinite quiver and the potential $W$ is the sum of the oriented 3-cycles in $Q$. 

\begin{rem} \emph{All 3-cycles in $Q$ are oriented 3-cycles} since the infinite quiver \eqref{infinite quiver} contains no unoriented 3-cycles.
\end{rem}

An example is given by the following quiver with relations.
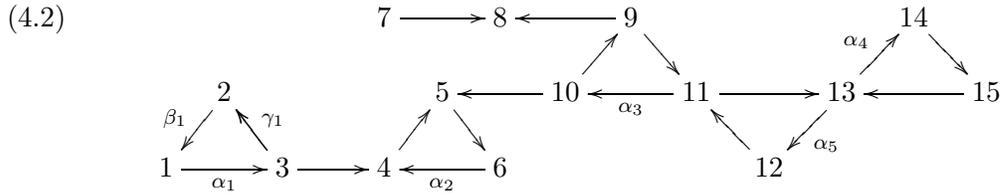
\begin{equation}\label{eq: cluster-tilted A15 example}
\xymatrixrowsep{15pt}\xymatrixcolsep{10pt}
\xymatrix{
  & & &   & 7\ar[rr]  &  &8 & &9\ar[dr]\ar[ll] & & &  &14\ar[dr]\\ 
&2\ar[ld]_{\beta_1} & & &&5\ar[dr] & &10\ar[ru]\ar[ll] && 
	11 \ar[ll]^{\alpha_3}\ar[rr]&
	 &13\ar[ur]^{\alpha_4} \ar[ld]^{\alpha_5}& &15\ar[ll]\\
1\ar[rr]_{\alpha_1} &&3\ar[rr]\ar[lu] \ar[lu]_{\gamma_1}& &4\ar[ru] & &6\ar[ll]^{\alpha_2} & & & 
	 &
	12\ar[ul]
	}
\end{equation}
\eqref{eq: cluster-tilted A15 example} is the quiver of a cluster-tilted algebra of type $A_{15}$. The potential $W$ is understood to be the sum of the five 3-cycles. Thus the relations which give $J(Q,W)$ are $\alpha_1\beta_1=\beta_1\gamma_1=\gamma_1\alpha_1=0$ and similarly for the other four 3-cycles. 

As explained in \cite[Sec 2.5]{A2}, cluster-tilted algebras of type $A_n$ are well-known to be ``gentle algebras'' as defined in \cite{BR} and thus all indecomposable modules are \emph{string modules} which means their supports are linear subquivers and their dimensions are equal to the size of their supports, i.e., they are 1-dimensional at each point in their support. An example of a string module for the quiver with potential \eqref{eq: cluster-tilted A15 example} is the 8-dimensional module with support at the 8 vertices 1,3,4,5,10,11,13,14. We refer to this module later as $X$. Note that the support of $X$ contains the arrows $\alpha_1,\alpha_3,\alpha_4$. One key property of string modules is that the full subquiver on which they are supported contains at most one arrow from each 3-cycle of $Q$.

By \cite{AH} \emph{iterated tilted} (or ``generalized tilted'') algebras are given by deleting one edge from each 3-cycle and retaining the relation that the remaining two arrows of the 3-cycles have zero composition. For example, if we delete the five arrow $\alpha_i$ in \eqref{eq: cluster-tilted A15 example} we obtain the following quiver with five zero relations $\beta_1\gamma_1=0$, etc, as indicated by the dotted lines.
\begin{equation}\label{eq: tilted A15 example}
\xymatrixrowsep{15pt}\xymatrixcolsep{10pt}
\xymatrix{
  & & &   & 7\ar[rr]  &  &8 & &9\ar[dr]\ar[ll] & & &  &14\ar[dr]\\ 
&2\ar[ld]_{\beta_1} & & &&5\ar[dr] & &10\ar[ru]\ar[ll] && 
	11 \ar@{--}[ll]\ar[rr]&
	 &13\ar@{--}[ur] \ar@{--}[ld]& &15\ar[ll]\\
1\ar@{--}[rr] &&3\ar[rr]\ar[lu]_{\gamma_1} \ar@{--}[lu]& &4\ar[ru] & &6\ar@{--}[ll] & & & 
	 &
	12\ar[ul]
	}
\end{equation}

In fact the algebra given by this particular quiver with relations is a \emph{tilted algebra} of type $A_n$ by the classification of such algebras from \cite{A} which we restate below (Theorem \ref{Assem's theorem}) using the ``diagram of a module'' which we now define.

Let $\Lambda=J(Q,W)$ be any cluster-tilted algebra of type $A_n$ as described above. Let $C$ be an associated iterated tilted algebra given by deleting one arrow from each 3-cycle of $Q$. Then $C\subset \Lambda$ and we have a forgetful functor $F:\Lambda\text-mod\to C\text-mod$ given by restriction of scalars. The functor $F$ sends an indecomposable $\Lambda$-modules $X$ to $FX=\bigoplus M_j$ where $M_j$ are indecomposable $C$-modules connected to each other by deleted arrows $\alpha_j$. For example, the $\Lambda$-module $X$ with support $\{1,3,4,5,10,11,13,14\}$ becomes $FX=M_1\oplus M_2\oplus M_3\oplus M_4$ where $M_1,M_4$ are the simple modules at vertices $1,14$ and $M_2,M_3$ are the $C$-modules with supports $\{3,4,5,10\}$ and $\{11,13\}$ respectively. 

\begin{defn}\label{def: diagram of a module}
We define the \emph{diagram} of an indecomposable $\Lambda$-module $X$ (with respect to an associated iterated tilted algebra $C$) to be the linear quiver with vertices labeled with the components $M_i$ of $FX$ and arrows given by the deleted arrows of $Q$ connecting vertices in the support of $X$. The deleted arrow $\alpha$ should connect $M_i$ to $M_j$ if $\alpha$ is an arrow in $Q$ from a vertex $v$ in the support of $M_i$ to a vertex $w$ in the support of $M_j$. 
\end{defn}

In this particular example, the diagram of $X$ is:
\[
	M_1\xrightarrow{\alpha_1} M_2\xleftarrow{\alpha_3} M_3\xrightarrow{\alpha_4} M_4
\]

Using the diagrams of the string modules for $\Lambda=J(Q,W)$, the classification of tilted algebras from \cite{A} can be stated as follows.

\begin{thm}[Assem]\label{Assem's theorem}
An iterated tilted algebra $C$ of type $A_n$ given by deleting arrows from the quiver $Q$ of a cluster-tilted algebra $\Lambda=J(Q,W)$ is a tilted algebra of type $A_n$ if and only if, for every indecomposable $\Lambda$-module $X$, the diagram of $X$ with respect to $C$ has alternating orientation of its arrows.
\end{thm}

The example given by \eqref{eq: tilted A15 example} is a tilted algebra since the diagrams of the string modules $X$ (given above) and $Y,Z$ with supports $\{1,3,4,6\}$ and $\{12,13,14\}$ resp. are:
\[
\bullet\xrightarrow{\alpha_1} \bullet\xleftarrow{\alpha_3} \bullet\xrightarrow{\alpha_4} \bullet
,\quad
	\bullet \xrightarrow{\alpha_1}\bullet\xleftarrow{\alpha_2} \bullet,\quad 
	\bullet \xleftarrow{\alpha_3} \bullet
	\xrightarrow{\alpha_4}\bullet .
\]

\subsection{Upper bound for lengths of MGSs}

For many cluster-tilted algebras of finite representation type the minimum length of a maximal green sequence has been computed \cite{CDRSW}, \cite{GMS}. In particular we have the following.

\begin{thm}\label{thm: minimum length}\cite{CDRSW}
For $\Lambda=J(Q,W)$ a cluster-tilted algebra of type $A_n$ the minimum length of a maximal green sequence is $n+k$ where $k$ is the number of 3-cycles in the quiver $Q$.
\end{thm}

Using the equivalent formulation of MGSs in terms of FHO sequences of modules (Corollary \ref{cor: main thm of section 2}) we will obtain an upper bound for the maximum length of a MGS.

\begin{prop}\label{prop: maximal and minimal mgs}
Let $m,p$ be the maximum and minimum lengths of maximal green sequences for any Jacobian algebra $\Lambda=J(Q,W)$ of finite representation type. Then $m+p-n$ is at most equal to the number of isomorphism classes of indecomposable $\Lambda$-modules.  
\end{prop}

\begin{proof} Let $M_1,\cdots,M_m$ be a FHO sequence of Schurian modules of maximal length. Construct another sequence of the $n$ simple $\Lambda$-modules in reverse order as they appear in the sequence $(M_i)$. Since $\Lambda$ has finite representation type, we can extend this to a complete FHO sequence $N_1,\cdots,N_q$ by inserting modules into the sequence. Then $q\ge p$ by definition of $p$.

Claim: A module $X$ appears in both lists if and only if $X$ is simple. 

Pf: Suppose $X$ is not simple. Let $S_j$ be a simple submodule of $X$ and $S_k$ a simple quotient module of $X$. Since $X$ is Schurian, $S_j\neq S_k$. If $X=M_s$ then $S_k<M_s<S_j$ in the first FHO sequence $(M_i)$. If $X=N_t$ then $S_k<N_t<S_j$ in $(N_j)$. This is not possible since $S_k,S_j$ are in the opposite order in the two sequences.

From this it follows that the union of the sets $\{M_i\}$ and $\{N_j\}$ has $m+q-n\ge m+p-n$ elements, proving the Proposition.
\end{proof}

Here is another proof of the Claim in the proof of Proposition \ref{prop: maximal and minimal mgs} using generic green paths. Let $\gamma$ be a generic green path crossing the walls $D(M_1),\cdots,D(M_m)$ in that order. Let $\gamma'$ be another generic green path crossing the walls $D(N_1),\cdots,D(N_q)$. By Theorem \ref{main thm of section three} such paths exist. Suppose that $\gamma(t_0)\in D(X)$. If $S_j\subset X$ then 
\[
	\gamma(t_0)\cdot \undim S_j=\gamma(t_0)_j<0
\]
and similarly, $\gamma(t_0)_k>0$ for all simple quotients $S_k$ of $X$. So, $\gamma$ crosses the hyperplane $D(S_k)=H(S_k)$ before time $t_0$ and will cross $D(S_j)=H(S_j)$ afterwards. Since $\gamma'$ crosses these hyperplanes in the reverse order, $\gamma'$ does not meet $D(X)$. Therefore, the paths $\gamma,\gamma'$ cannot cross the same semistability set $D(X)$ except for the hyperplanes $D(S_i)=H(S_i)$.

\begin{cor}\label{cor: upper bound for mgs of An}
If $\Lambda=J(Q,W)$ is cluster-tilted of type $A_n$, the length of a maximal green sequence is at most $\binom {n+1}2-k$ where $k$ is the number of oriented 3-cycles in the quiver $Q$.
\end{cor}

\begin{proof}
By Theorem \ref{thm: minimum length} the minimum length of a maximal green sequence is $p=n+k$. Since every cluster-tilted algebras of type $A_n$ has exactly $\binom{n+1}2$ indecomposable modules up to isomorphism, we have, by Proposition \ref{prop: maximal and minimal mgs}, that
\[
	m+p-n=m+k\le \binom{n+1}2
\]
where $m$ is the maximum length of a maximal green sequence. Thus $m\le \binom{n+1}2-k$.
\end{proof}

\begin{eg}\label{eg: A5 tilted algebra example}
Consider the cluster-tilted algebra $\Lambda=J(Q,W)$ where $Q$ is the quiver below (on the left) and $W$ is given by the two 3-cycles. This is cluster-tilted of type $A_5$ and therefore has $\binom62=15$ modules.
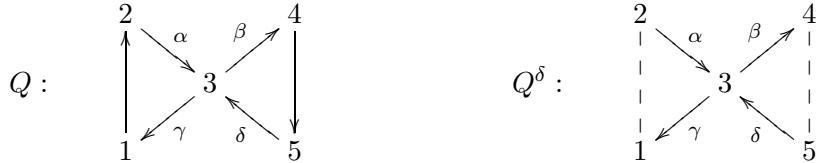
\[
\xymatrixrowsep{10pt}\xymatrixcolsep{20pt}
\xymatrix{
&2\ar[dr]^\alpha &&4\ar[dd] &&&& 2\ar[dr]^\alpha &&4\\
Q:&& 3\ar[ru]^\beta\ar[ld]^\gamma&&&&Q^\delta:&& 3\ar[ru]^\beta\ar[ld]^\gamma\\
&1\ar[uu] && 5\ar[lu]^\delta&&&&1\ar@{--}[uu] && 5\ar[lu]^\delta\ar@{--}[uu]
	}
\]
If we remove the two unlabeled arrows in $Q$ we are left with the quiver $Q^\delta$ indicated above on the right, with the relations $\gamma\alpha=0=\beta\delta$. By Assem's criterion \ref{Assem's theorem}, this is the quiver with relations of a tilted algebra $C$. The Auslander-Reiten quiver of $C$ is given by:
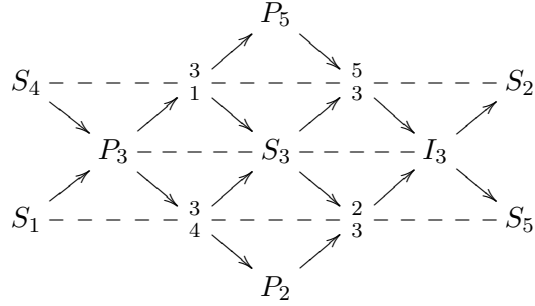
\[
\xymatrixrowsep{8pt}\xymatrixcolsep{15pt}
\xymatrix{
&&&P_5\ar[rd]\\
S_4\ar[dr]\ar@{--}[rr] &&3\atop1 \ar@{--}[rr]\ar[dr]\ar[ur] && 5\atop3\ar[dr]\ar@{--}[rr]&&S_2\\
& P_3\ar@{--}[rr]\ar[dr]\ar[ur] && S_3\ar@{--}[rr]\ar[dr]\ar[ur] && I_3\ar[dr]\ar[ur] \\
S_1\ar[ur]\ar@{--}[rr] && 3\atop4\ar@{--}[rr]\ar[dr]\ar[ur] &&2\atop3 \ar[ur]\ar@{--}[rr]&&S_5\\
&&&P_2\ar[ru]
	}
\]
This is a full subquiver of the Auslander-Reiten quiver of $\Lambda$ with only two modules missing: the projective $\Lambda$-modules $P_1$ and $P_4$. Consider the partial ordering on this set of 13 modules given by $M_1\prec M_2$ if there is an oriented path in this AR-quiver from $M_1$ to $M_2$. For example, $S_1,S_4$ are minimal in this partial ordering and $S_2,S_5$ are maximal. We refer to this as the \emph{partial ordering given by the AR-quiver}. Any refinement of this partial ordering to a total ordering will be called \emph{one of the total orderings given by the AR-quiver}.

Since $\Hom_\Lambda(M_1,M_2)\neq 0$ implies $\Hom_C(M_1,M_2)\neq 0$ implies $M_1\preceq M_2$, weakly FHO sequences for both $\Lambda$ and $C$ with $13$ objects can be given by arranging these 13 objects in the \emph{reverse order of any total ordering given by the AR-quiver}. In any such sequence one of $S_2,S_5$ will be first and one of $S_1,S_4$ will be last. These maximal green sequences for $\Lambda$ have the maximum length since, by Corollary \ref{cor: upper bound for mgs of An}, the length of any MGS is at most $15-k=13$. Therefore, they are complete FHO sequences.

Putting the simple modules in reverse order (with $S_1,S_4$ first, in either order, and $S_2,S_5$ last, in either order) and inserting the missing modules $P_1,P_4$ we get the minimum length complete FHO sequence in $\Lambda\text-mod$: $S_1,S_4,P_1,S_3,P_4,S_2,S_5$ of length $n+k=5+2=7$.
\end{eg}

Our second example, from \cite{AI}, is a cluster-tilted algebra of type $D_n$ (for $n\ge4$). In this example, we use the following terminology. Two MGSs as said to be \emph{equivalent} if they have the same set of modules in their corresponding complete FHO sequences, i.e., the corresponding sequences of $c$-vectors are permutations of each other.

\begin{eg}\label{eg: Dn}\cite{AI}
Let $Q_n$ be the oriented cycle quiver with $n$ vertices and let $\Lambda_n$ be $KQ_n$ module $rad^{n-1}KQ_n$. For $n=4$ the quiver is:
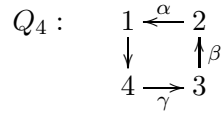
\[
\xymatrixrowsep{10pt}\xymatrixcolsep{15pt}
\xymatrix{
Q_4: &1\ar[d]& 2\ar[l]_\alpha \\
&4\ar[r]_\gamma & 3\ar[u]_\beta	}
\]
This cluster-tilted algebra is the relation-extension of the tilted algebra $C$ given by the subquiver
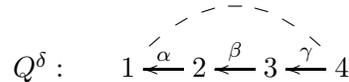
\[
\xymatrixrowsep{15pt}\xymatrixcolsep{15pt}
\xymatrix{
	Q^\delta:&1\ar@{--}@/^2pc/[rrr] &2\ar[l]_{\alpha}&3\ar[l]_{\beta}&4\ar[l]_{\gamma}
	}
\]
of $Q_4$ modulo the relation $\alpha\beta\gamma=0$. The Auslander-Reiten quiver is given by:
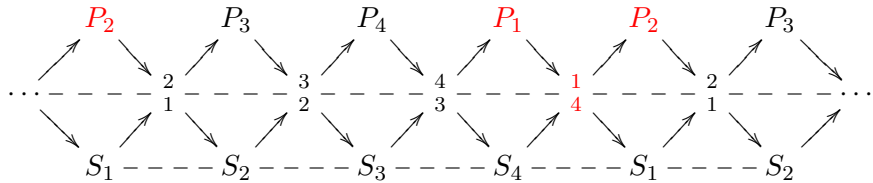
\[
\xymatrixrowsep{10pt}\xymatrixcolsep{10pt}
\xymatrix{
&\color{red}P_2\ar[dr]&& P_3 \ar[dr]&& P_4 \ar[dr]&& \color{red}P_1 \ar[dr]&& \color{red}P_2\ar[dr] && P_3\ar[dr] && \\
\cdots\ar[ru]\ar[rd]\ar@{--}[rr]&&2\atop1\ar@{--}[rr]\ar[ru]\ar[rd]&& 3\atop2\ar[ru]\ar@{--}[rr]\ar[rd]&& 4\atop3\ar@{--}[rr]\ar[ru]\ar[rd]&& \color{red}1\atop4\ar@{--}[rr]\ar[ru]\ar[rd]&& 2\atop1\ar@{--}[rr]\ar[ru]\ar[rd]&&\cdots \\
&S_1\ar[ru]\ar@{--}[rr]&& S_2\ar[ru]\ar@{--}[rr]&&S_3\ar[ru]\ar@{--}[rr]&&S_4\ar[ru]\ar@{--}[rr]&&S_1\ar[ru]\ar@{--}[rr]&&S_2\ar[ru]
	}
\]
where five objects are repeated. There are 12 indecomposable $\Lambda$-modules, but only the 9 black objects are modules over the tilted algebra $C=KQ^\delta$. These objects, in the following order, form a complete FHO sequence for $\Lambda$.
\[
	S_4,{\,^4_3}, P_4,S_3, {\,^3_2}, P_3,S_2,{\,^2_1},S_1.
\]
This set of modules can be arranged in four different ways since $P_4,S_3$ and $P_3,S_2$ can be taken in either order. This gives four equivalence MGS's. Similarly, there are three other equivalence classes of maximal green sequences of length 9 given by deleting $P_2,{\,^2_1},P_3$ or $P_3,{\,^3_2},P_4$ or $P_4,{\,^4_3},P_1$. This is an example of the following theorem.
\end{eg}

\begin{thm}\cite{AI}
The quiver $Q_n$ given by a single oriented $n$-cycle has, up to permutation of $c$-vectors, $n$ maximal green sequences of maximal length and these all have length $\binom n2+n-1$.
\end{thm}

\subsection{Lower bound for maximal length of MGSs}

The two examples above illustrate the following lower bound for the maximal length of maximal green sequences for a cluster-tilted algebra of finite representation type. We use the fact that every cluster-tilted algebra $\Lambda$ is the relation-extension of a tilted algebra $C$ which is usually not uniquely determined.

\begin{lem}\label{lem: lower bound for MGSs}
Let $\Lambda=J(Q,W)$ be a cluster-tilted algebra of finite representation type which is the relation-extension of a tilted algebra $C$. Then $Q$ has a maximal green sequence of length greater than or equal to the number of indecomposable $C$-modules.
\end{lem}

\begin{proof}
Since $C$ is tilted of finite representation type it is derived equivalent to a hereditary algebra of finite representation type. Since the bounded derived category of such an algebra has no oriented cycles, the Auslander-Reiten quiver of $C$ has no oriented cycles. (More generally, the AR-quiver of any tilted algebra has an acyclic component \cite{ASS}, VIII.3.5.) Let $M_i$ be the indecomposable $C$-modules arranged in the reverse order of one of the total orderings given by this AR-quiver. Then $\Hom_{C}(M_i,M_j)=0$ for $i<j$. Since the quiver of $C$ is a subquiver of the quiver of $\Lambda$ modulo all relations which are supported on that subquiver, $C\text-mod$ is exactly embedded as a full subcategory of $\Lambda\text-mod$. Thus $\Hom_\Lambda(M_i,M_j)=0$ for $i<j$. Since $\Lambda$ has finite representation type, the sequence $(M_i)$ can be completed to a maximal (complete) FHO sequence which corresponds to a maximal green sequence for $Q$ of the same length by Corollary \ref{cor: main thm of section 2}.
\end{proof}

For the next theorem we need the following lemma about the diagram of a $\Lambda$-module (Definition \ref{def: diagram of a module}) with respect to an associated iterated tilted algebra $C$.

\begin{lem}\label{lem: arrow in diagram gives path in C-mod}
Let $\Lambda=J(Q,W)$ be a cluster-tilted algebra of type $A_n$ and let $C$ be an iterated tilted algebra obtained by deleting one arrow from each 3-cycle in $Q$. Let $X$ be a $\Lambda$ module whose diagram with respect to $C$ is $M_1\xrightarrow\alpha M_2$. Then, there is a sequence of nonzero morphisms of $C$-modules $M_1\to E_1\to E_2\to M_2$.
\end{lem}

\begin{proof} The quiver $Q^\delta$ of $C$ is obtained from the quiver $Q$ of $\Lambda$ by deleting one arrow from each 3-cycle in $Q$. Since the $\Lambda$-modules $X$ has diagram $M_1\xrightarrow\alpha M_2$, there is a 3-cycle
\[
\xymatrixrowsep{15pt}\xymatrixcolsep{10pt}
\xymatrix{
& w\ar[dl]_\beta\\
v_1\ar[rr]^\alpha &&v_2\ar[lu]_\gamma
	}
\]
in $Q$ so that $v_1,v_2$ are vertices in the support of $M_1,M_2$ respectively and $\alpha$ is a deleted arrow. Since $X$ is a string module, the vertex $w$ is not in the support of $X$.

Since $\alpha$ is a deleted arrow, the other two arrows in the 3-cycle $\beta,\gamma$ are within the quiver $Q^\delta$ of $C$. Therefore, $\Ext_{C}(S_w,M_1)$ and $\Ext_{C}(M_2,S_w)$ are nonzero where $S_w$ is the simple $C$-module supported at $w$. Letting $E_1,E_2$ be the extensions: $M_1\to E_1\to S_w$ and $S_w\to E_2\to M_2$ we obtain the chain of nonzero morphisms in $C\text-mod$:
\[
	M_1\to E_1\to E_2\to M_2
\]
as claimed.
\end{proof}

We also need the following well-known fact.

\begin{lem}\cite{AH}\label{lem: AR of C is simply-connected}
The Auslander-Reiten quiver of any iterated tilted algebra of type $A_n$ has no oriented cycles.
\end{lem}


\begin{thm}\label{thm: MSG for An}
Let $C$ be a tilted algebra of type $A_n$ and let $\Lambda$ be the relation-extension of $C$. Then the indecomposable $C$-modules arranged in the reverse order of one of the total orderings given by its Auslander-Reiten quiver form a complete FHO sequence of $\Lambda$-modules.
\end{thm}

\begin{proof}
Let $\Lambda=J(Q,W)$ and let $Q^\delta$ be the quiver of $C$. Since the AR-quiver of $C$ contains no oriented cycles, the indecomposable $C$-modules can be arranged in a weakly FHO sequence. We claim that any such sequence is maximal and thus complete.

To prove this claim, let $X$ be any indecomposable $\Lambda$-module which is not a $C$-module. Then $X$ is a string module having at least one deleted arrow in its support. In other words, the diagram of $X$ has at least one arrow
\[
	\cdots \to\bullet\leftarrow M_j\xrightarrow\alpha M_{j+1}\leftarrow \bullet\to\cdots
\]

By Assem's Theorem \ref{Assem's theorem}, the arrows in the diagram of $X$ alternate in orientation. Thus $M_j$ is a source and $M_{j+1}$ is a sink. This implies that we have nonzero morphisms $M_{j+1}\to X$ and $X\to M_j$. In order to insert $X$ into the weakly FHO sequence of indecomposable $C$-modules, the $C$-module $M_j$ must come before $M_{j+1}$ in the sequence since $X$ cannot come before $M_j$ or after $M_{j+1}$.

By Lemma \ref{lem: arrow in diagram gives path in C-mod}, there is a sequence of nonzero morphisms of $C$-modules $M_j\to E_1\to E_2\to M_{j+1}$. Therefore, in any weakly FHO sequence containing all indecomposable $C$-modules, $M_j$ must come after $M_{j+1}$. Thus $X$ cannot be added to this sequence. So, the sequence is complete as claimed.
\end{proof}

Theorem \ref{thm: MSG for An} extends easily to the case when $C$ is iterated tilted of type $A_n$.

\begin{cor}\label{MSG for iterated tilted}
Let $\Lambda=J(Q,W)$ be a cluster-tilted algebra of type $A_n$. Let $C$ be an associated iterated tilted algebra given by deleting one arrow from each 3-cycle of $Q$. Then the indecomposable $C$-modules can be arranged into a complete FHO sequence for $\Lambda$.
\end{cor}

\begin{proof}
The proof is the same as for Theorem \ref{thm: MSG for An} with the exception that the diagram of a $\Lambda$-module $X$ might not be alternating. In that case, we take the longest sequence of arrows in the diagram of $X$ which are pointing in the same direction, say
\[
	\cdots \to\bullet\leftarrow M_j\xrightarrow{\alpha_j}M_{j+1}\xrightarrow{\alpha_{j+1}}M_{j+2}\rightarrow\cdots\rightarrow M_k\leftarrow \bullet\to\cdots.
\]
Again, we have nonzero morphisms $M_k\to X$ and $X\to M_j$. By Lemma \ref{lem: arrow in diagram gives path in C-mod}, there are sequences of nonzero morphism in $C\text-mod$:
\[
	M_j\to E_1\to E_2\to M_{j+1}\to E_3\to \cdots \to M_k
\]
So, $M_j$ must comes after $M_k$ in any weakly FHO sequence of all indecomposable $C$-modules and $X$ cannot be inserted into such a sequence.
\end{proof}

\subsection{Conjecture} 

Examination of Example \ref{eg: A5 tilted algebra example} with all possible cut sets suggests that iterated tilted algebras will have fewer modules and that the maximum length MGS is likely to be given by a tilted algebra. This lead to the following conjecture about the size and nature of maximal green sequences of maximal length for cluster-tilted algebras.

\begin{conj}[I-Todorov]\label{conjectural description of maximal MGSs}
Let $\Lambda=J(Q,W)$ be a cluster-tilted algebra of finite representation type and let $C$ be one of the associated tilted algebras.
\begin{enumerate}
\item[(a)] The indecomposable $C$-modules can be arranged to form a complete FHO sequence of $\Lambda$-modules whose dimension vectors, by Corollary \ref{cor: main thm of section 2}, form the $c$-vectors of a maximal green sequence for $Q$.
\item[(b)] The longest maximal green sequence for $Q$ is given in this way.
\end{enumerate}
\end{conj}

Theorem \ref{thm: MSG for An} proves \ref{conjectural description of maximal MGSs}(a) in type $A_n$. 


In type $A_n$, it is easy to count the number of indecomposable modules of any given tilted algebra. It is $\binom {n+1}2$ minus the number of representations of the quiver of the cluster-tilted algebra which are nonzero on some deleted arrow. The set of deleted arrows consists of one arrow from each oriented 3-cycle so that the resulting quiver with relations is tilted of type $A_n$ as described by Assem (Theorem \ref{Assem's theorem}).

\begin{eg}\label{eg: Garver and I}
This example was worked out in detail by Al Garver and K.Igusa (unpublished). What follows is a streamlined argument using the results of this paper.
Let $\Lambda=J(Q,W)$ be the cluster-tilted algebra of type $A_9$ where $Q$ is the following quiver with either of the two possible orientations of each of the two inner 3-cycles.
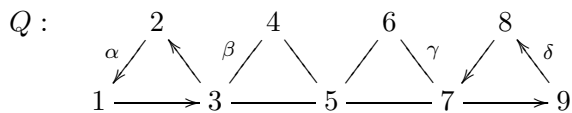
\[
\xymatrixrowsep{15pt}\xymatrixcolsep{10pt}
\xymatrix{
Q:&& 2\ar[dl]_{\alpha} && 4 && 6 &&8\ar[dl]\\
&1\ar[rr] &&3\ar[lu]\ar@{-}[ru]^\beta\ar@{-}[rr]&& 5\ar@{-}[ru]\ar@{-}[rr]\ar@{-}[lu] && 7\ar[rr]\ar@{-}[lu]_\gamma && 9\ar[lu]_\delta
	}
\]
If we delete the four arrows $\alpha,\beta,\gamma,\delta$, the result is a tilted algebra of type $A_9$ regardless of orientation by Theorem \ref{Assem's theorem}. So, Theorem \ref{thm: MSG for An} applies. When we remove these arrows we remove 8 modules since $\alpha$, $\delta$ each support only one module and $\beta,\gamma$ each support three. (In all iterated tilted cases we remove more than 8.) These eight $\Lambda$-modules, denoted by their supports by $12,34,134,234,67,678,679,89$ appear in the list below. By Theorem \ref{thm: MSG for An} the other $\binom{10}2-8=37$ modules form a complete FHO sequence. We claim that $Q$ has no MGS of length greater than 37. The reason is that $\Lambda\text-mod$ has 8 disjoint oriented cycles of nonzero morphisms: 
\begin{enumerate}
\item $12\to 23\to 31\to 12$
\item $34 \leftrightarrow 45 \leftrightarrow 53 \leftrightarrow 34$ (orientation depending on the orientation of $Q$)
\item $134 \leftrightarrow 457\leftrightarrow 1357\leftrightarrow 134$
\item $234\leftrightarrow 4579 \leftrightarrow 23579\leftrightarrow234$
\item $ 67\leftrightarrow75\leftrightarrow56\leftrightarrow67$
\item $678\leftrightarrow 4578\leftrightarrow456\leftrightarrow678$
\item $679\leftrightarrow3579\leftrightarrow356\leftrightarrow679$
\item $89\to 97\to 78\to 89$
\end{enumerate}
To have a FHO sequence at least one module must be deleted from each of these 8 cycles. So, 37 is the strict upper bound for the length of any MGS for $Q$.
\end{eg}

\section*{Acknowledgements} A preliminary version of this paper was presented at the University of Connecticut in March 2017 and at Idun Reiten's birthday conference in Trondheim in May 2017. The first author thanks Osamu Iyama for his very nice talk on a related topic \cite{DIJ} in Trondheim and for some very useful conversations. He also benefitted from several conversation with Hipolito Treffinger and Al Garver. The second author thanks Ibrahim Assem for his very nice notes about cluster-tilted algebras \cite{A2} which was the inspiration for the conjectural structure of maximal green sequences proposed in the Appendix. We thank the anonymous referees for their numerous very helpful suggestions. Finally, we thank the organizers and participants of the meeting of the Canadian Math Society in Fredericton, NB in June 2018 for their comments and encouragement.

\end{document}

Proposition 2.1 is still \ref{prop: HN stratification with 3 strata}

Corollary 2.3 is still \ref{cor: max weakly forward hom-orth is complete}

Lemma 2.4 is now \ref{lem:when Sk is not one of the Mi}

Lemma 2.6 is now \ref{lem: inductive step A=B for Q finite type}

Lemma 2.7 is now \ref{lem: rotation lemma for modules}

New Proposition 2.7 is now \ref{prop: formula for mutation of W}

Theorem 2.8 is now Theorem \ref{thm: inductive thm implying main thm}

Cor 2.9 is now \ref{cor: main thm of section 2}

Cor 2.10 (rotation lemma) is now \ref{cor: rotation lemma}

Proposition 2.11 is now \ref{prop: iterated mutation}

Remark 3.2 is still \ref{added remark}

Definition 3.4 is now \ref{def: green path}

Corollary 3.7 is now \ref{cor: first wall is simple}

3.16 is now \ref{cor: formerly 3.16}

3.19 is now \ref{prop: maximal and minimal mgs}

Corollary 3.20 is now \ref{cor: upper bound for mgs of An}

Example 3.21 is now \ref{eg: A5 tilted algebra example}

Example 3.22 is now \ref{eg: Dn}

Lemma 3.24 is now \ref{lem: lower bound for MGSs}

Theorem 3.24 is now 

Theorem 3.25 is now \ref{thm: MSG for An}

Conjecture 3.28 is now \ref{conjectural description of maximal MGSs}

Example 3.29 is now \ref{eg: Garver and I}

\end{document}

